\documentclass[12pt]{amsart}
\usepackage{latexsym,amscd, amsmath, amssymb,epsfig,xypic,tikz, array} 
\usetikzlibrary{patterns,arrows,decorations.pathreplacing}

\theoremstyle{plain}
  \newtheorem{theorem}{Theorem}[section]
  \newtheorem{proposition}[theorem]{Proposition}
  \newtheorem{lemma}[theorem]{Lemma}
  \newtheorem{corollary}[theorem]{Corollary}
  
\theoremstyle{definition}
  
  \newtheorem{example}[theorem]{Example}

  \newtheorem{hypothesis}[theorem]{Hypothesis}
\theoremstyle{remark}

\numberwithin{equation}{section}

\def\tempbaselines
{\baselineskip14pt\lineskip2pt
   \lineskiplimit2pt}
\def\diagram#1{\null\,\vcenter{\tempbaselines
\mathsurround=0pt
    \ialign{\hfil$##$\hfil&&\quad\hfil$##$\hfil\crcr
      \mathstrut\crcr\noalign{\kern-\baselineskip}
  #1\crcr\mathstrut\crcr\noalign{\kern-\baselineskip}}}\,}
\def\clap#1{\hbox to 0pt{\hss$#1$\hss}}

\def\calC{{\mathcal C}}

\def\calI{{\mathcal I}}

\def\calQ{{\mathcal Q}}
\def\calP{{\mathcal P}}
\def\calS{{\mathcal S}}

\def \End{\mathop{\rm End}\nolimits}

\def \Hom{\mathop{\rm Hom}\nolimits}

\def \Rad{\mathop{\rm Rad}\nolimits}

\def \Soc{\mathop{\rm Soc}\nolimits}

\def\ZZ{{\Bbb Z}}

\begin{document}

\title[The Bounded Derived Category of a Poset]
{The Bounded Derived Category of a Poset}

\author{Kosmas Diveris}\email{diveris@stolaf.edu}\address{Department of Mathematics\\
St. Olaf College\\ Northfield, MN 55057, USA} 
\author{Marju Purin}\email{purin@stolaf.edu}\address{Department of Mathematics\\
St. Olaf College\\ Northfield, MN 55057, USA}
\author{Peter Webb}
\email{webb@math.umn.edu}
\address{School of Mathematics\\
University of Minnesota\\
Minneapolis, MN 55455, USA}



\begin{abstract}
We introduce a new combinatorial condition on a subinterval of a poset $\calP$ (a \textit{clamped} subinterval) that allows us to relate the Auslander-Reiten quiver of the bounded derived category of $\calP$ to that of the subinterval. Applications include the determination of when a poset is fractionally Calabi-Yau and the computation of the Auslander-Reiten quivers of both the bounded derived category and of the module category.
\end{abstract}

\maketitle

\section{Introduction}

We study representations of a finite poset $\calP$ over a field $k$. By this we mean representations of the incidence algebra $k\calP$; they may also be identified as functors from $\calP$ to the category of vector spaces over $k$. 

Such representations have been an object of study since at least the pioneering work of Mitchell \cite{Mit1, Mit2}. A notable result was the determination of the posets of finite representation type by Loupias~\cite{Lou}. More recent work has centered around the abstract properties of these representations, including the question of when the bounded derived category of an incidence algebra is fractionally Calabi-Yau \cite{Cha, HI}. Examples show that such questions do not seem to have a particularly straightforward answer in terms of the combinatorics of the poset. Our main goal in this paper is to introduce a technique that allows us to describe the representation theory of a wide class of posets exactly by such combinatorial means.

We develop a new method, which we call \textit{clamping}, to describe the Auslander-Reiten quiver components of $k\calP$ in both the module category and the derived category. We show that the presence of a subposet of $\calP$ of a kind that we call a \textit{clamped interval} allows us to copy parts of the Auslander-Reiten quiver of the derived category of the clamped interval directly into the Auslander-Reiten quiver for $k\calP$. When the poset is built up entirely from clamped intervals it allows us to compute the shape of the Auslander-Reiten quiver of $D^b(k\calP)$ very quickly, simply from the combinatorics of the poset. In this situation, it turns out that $k\calP$ is derived equivalent to a hereditary algebra (i.e. $k\calP$ is \textit{piecewise hereditary}). In other cases, when we might not have such complete information, the method provides a way to show that certain bounded derived categories are not fractionally Calabi-Yau. 

In the context of representations of finite dimensional algebras, we explore the connection between Auslander-Reiten triangles in the bounded derived category and Auslander-Reiten sequences in the module  category, developing the fundamentals of this relationship. Applying this to posets built up out of clamped intervals, we use the approach to compute the shape of the Auslander-Reiten quiver of the module category. This allows us to construct posets with various properties: for each tree that is a barycentric subdivision, we are able to construct a lattice that is derived equivalent to the path algebra of any quiver that is an orientation of the given tree.  In another example, we exhibit a lattice of finite representation type that is derived equivalent to a hereditary algebra of wild representation type; and we indicate other examples along these lines. The method provides a criterion for a poset constructed out of clamped intervals to have finite representation type, and we use this to list all such posets of finite representation type.

We present the definition and basic properties of clamping in Section 2. The main features are that restriction, induction, and coinduction between the clamped interval and the whole poset behave particularly well, allowing us to copy resolutions from the clamped interval to the whole poset, and also to copy regions of the Auslander-Reiten quiver in a similar way.

In Section 3 we use this in an approach that shows that many posets have bounded derived categories that are not fractionally Calabi-Yau. 

In Section 4 we define the family of posets obtained iteratively out of clamping, denoting the class of posets obtained in this way $\mathcal{IC}$.  We use the results from Section 2 to describe the Auslander-Reiten quiver of the derived category for such posets, using the combinatorics of the poset. We prove that such algebras $k\calP$ are piecewise hereditary and characterize the trees that arise as their ordinary quivers. 

We start Section 5 by establishing a fundamental relationship between Auslander-Reiten triangles in the derived category and  Auslander-Reiten sequences in the module category, in the context of representations of a finite-dimensional algebra. Although the results are straightforward, we have not so far found them in the literature, and they have broader application than this paper.
This, along with clamping theory, allows us to compute  Auslander-Reiten quiver components of the module categories of $\mathcal{IC}$ poset algebras.

We conclude in Section 6 by determining the Auslander-Reiten quivers of $\mathcal{IC}$ posets of finite representation type. 

We assume familiarity with the methods of Auslander-Reiten theory, including the construction of Auslander-Reiten sequences and triangles. For the module category we refer to the books \cite{ASS, ARS} for this, and for triangulated categories we refer to \cite{Hap}. For general aspects of representations of categories, and posets in particular, we refer to \cite{Web1} for background.

\section{Clamped intervals}

We consider representations of a finite poset $\calP$ over a field $k$ and introduce the notation we will use. 
By a representation of $\calP$ over $k$ we  mean a functor from $\calP$ to the category of $k$-vector spaces, regarding $\calP$ as a category. Such representations may be identified as modules for the \textit{category algebra} (see~\cite{Web1}) or \textit{incidence algebra} of $\calP$, which we denote by $k\calP$.
Whenever $\calQ$ is a subposet of $\calP$ that is \textit{convex} (meaning that $a,b\in \calQ$ and $a\le x\le b$ imply $x\in \calQ$) we let $k_\calQ$ be the representation of $\calP$ that takes the value $k$ on every element of $\calQ$ and that is 0 elsewhere, with identity maps corresponding to each comparison in $\calQ$.
If $M$ is a representation of $\calP$ the \textit{support} of $M$ is the set of elements $x\in\calP$ for which $M(x)\ne 0$; thus the support of $k_\calQ$ is $\calQ$.
The fact that $k_\calQ$ is a representation is part of the general observation that whenever $U$ is a $k\calQ$-module we may regard it also as a $k\calP$-module with support in $\calQ$ by letting morphisms of $\calP$ that are not in $\calQ$ act as zero. 
Thus $k_\calQ$ is the \textit{constant functor} on $\calQ$, regarded as a $k\calP$-module in this way. 

For each element $x\in\calP$ there is an indecomposable projective representation of $\calP$ generated by its value at $x$, namely the representable functor at $x$. We will write it as $P_x$, or as $P_x^\calP$ if we wish to emphasize the poset of which this is a representation. 
The corresponding indecomposable injective representation with socle at $x$ is denoted $I_x$ or $I_x^\calP$. Thus $P_x=k_{[x,\infty)}$ and $I_x=k_{(-\infty,x]}$.

For each subposet $\calQ$ of $\calP$ we define functors
$$
\uparrow_\calQ^\calP:k\calQ\hbox{-mod}\to k\calP\hbox{-mod}
\quad\hbox{and}\quad
\Uparrow_\calQ^\calP:k\calQ\hbox{-mod}\to k\calP\hbox{-mod}
$$
by $U\uparrow_\calQ^\calP=k\calP\otimes_{k\calQ}U$ and $U\Uparrow_\calQ^\calP=\Hom_{k\calQ}(k\calP,U)$. These are the \textit{induction} and \textit{coinduction} to $\calP$ of representations of $\calQ$. We will also make use of the corresponding functors between bounded derived categories
$$
\uparrow_\calQ^\calP:D^b(k\calQ)\to D^b(k\calP)
\quad\hbox{and}\quad
\Uparrow_\calQ^\calP:D^b(k\calQ)\to D^b(k\calP)
$$
On an object $M$ of $D^b(k\calQ)$, namely a complex of $k\calQ$-modules with nonzero homology in only finitely many degrees, 
we define $\uparrow_\calQ^\calP$ as the left derived functor
$$
M\uparrow_\calQ^\calP=k\calP\otimes_{k\calQ}^LM
$$
and $\Uparrow_\calQ^\calP$ as the right derived functor
$$
M\Uparrow_\calQ^\calP=R\Hom_{k\calQ}(k\calP,M).
$$
The way these functors are computed on a complex $M$ is that we take a complex of projective modules isomorphic to $M$ (a projective resolution) and apply $\uparrow_\calQ^\calP$ term by term to get $M\uparrow_\calQ^\calP$. To get $M\Uparrow_\calQ^\calP$ we take a complex of injective modules isomorphic to $M$ (an injective resolution) and apply $\Uparrow_\calQ^\calP$ term by term.
In both the module and derived category cases, $\uparrow_\calQ^\calP$ is the left adjoint of the restriction $\downarrow_\calQ^\calP$, and $\Uparrow_\calQ^\calP$ is the right adjoint of $\downarrow_\calQ^\calP$. 
We say that a module (resp. complex) $M$ is \textit{induced} from $\calQ$ if it has the form $N\uparrow_\calQ^\calP$ for some module (resp. complex) $N\in D^b(k\calQ)$ and \textit{coinduced} from $\calQ$ if it has the form $N\Uparrow_\calQ^\calP$ for some module (resp. complex) $N\in D^b(k\calQ)$.

We now come to an important new definition that specifies the way an interval may be positioned in a poset.
Given elements $a\le b$ in a finite poset $\calP$, we say that the closed interval $[a,b]$ is \textit{clamped} in $\calP$ if $a\le x$ implies either $x\le b$ or $b\le x$, and $y\le b$ implies either $a\le y$ or $y\le a$. For example, in the following posets the interval $[a,b]$ is clamped in the first poset, but no interval apart from single points or the whole poset is clamped in the second one. 
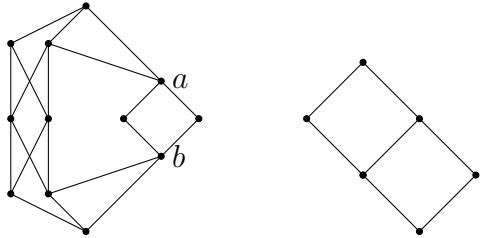
\begin{center}
\begin{minipage}[b]{0.3\linewidth}
\begin{tikzpicture}[xscale=.5,yscale=.5]
\draw (2,0)--(0,1)--(0,3)--(0,5)--(2,6)--(4,4)--(5,3)--(4,2)--(2,0);
\draw (2,0)--(1,1)--(1,3)--(1,5)--(2,6);
\draw (1,1)--(4,2)--(3,3)--(4,4)--(1,5);
\draw (1,1)--(0,3)--(1,5);
\draw (0,1)--(1,3)--(0,5);
\draw[fill] (2,0) circle [radius=0.08];
\draw[fill] (0,1) circle [radius=0.08];
\draw[fill] (1,1) circle [radius=0.08];
\draw[fill] (4,2) circle [radius=0.08];
\draw[fill] (0,3) circle [radius=0.08];
\draw[fill] (1,3) circle [radius=0.08];
\draw[fill] (3,3) circle [radius=0.08];
\draw[fill] (5,3) circle [radius=0.08];
\draw[fill] (4,4) circle [radius=0.08];
\draw[fill] (0,5) circle [radius=0.08];
\draw[fill] (1,5) circle [radius=0.08];
\draw[fill] (2,6) circle [radius=0.08];

\node [ right] at  (4,4)  {$a$};
\node [ right] at  (4,2)   {$b$};
\end{tikzpicture}
\end{minipage}
\begin{minipage}[b]{0.3\linewidth}
\begin{tikzpicture}[xscale=.5,yscale=.5]
\draw (4,0)--(2.5,1.5)--(4,3)--(5.5,1.5)--(4,0);
\draw[fill] (4,0) circle [radius=0.08];
\draw[fill] (2.5,1.5) circle [radius=0.08];
\draw[fill] (2.5,4.5) circle [radius=0.08];
\draw[fill] (4,3) circle [radius=0.08];
\draw[fill] (5.5,1.5) circle [radius=0.08];

\draw (4,3)--(2.5,4.5)--(1,3)--(2.5,1.5);
\draw[fill] (1,3) circle [radius=0.08];
\draw[fill] (2.5,1.5) circle [radius=0.08];
\end{tikzpicture}
\end{minipage}
\end{center}

Our convention when drawing Hasse diagrams of posets is that smaller elements are positioned above larger elements, which is the opposite of the usual convention. We do this because when drawing diagrams of representations it is usual to put elements that generate a module at the top of the module. 

We will show that the representation theory of a clamped interval reappears in the representation theory of $\calP$ in a rather transparent way. We will start by considering the implications of this condition for the derived category, because in this setting induction and coinduction surprisingly coincide on the interior of a clamped interval. We write half-open intervals in $\calP$ as $[a,b)$ and $(a,b]$. We mention that some parts of the next proposition follow from Section IX.9 of Mitchell's text~\cite{Mit1}, since clamped intervals are normal subsets in Mitchell's terminology.

\begin{proposition}\label{clamping-adjoint-properties}
Let the interval $[a,b]$ be clamped in the poset $\calP$. Let $M$ be a complex of $k\calP$-modules with homology supported on $[a,b)$ and let $N$ be a complex of $k\calP$-modules with homology supported on $(a,b]$.
\begin{enumerate}
\item The minimal projective resolution $\Pi_M$ of $M$ over $k\calP$ only has terms $P_x$ with $x\in[a,b]$, and is induced from $[a,b]$.
\item The adjunction counit $M\downarrow_{[a,b]}\uparrow^{\calP}\to M$ is an isomorphism, so that $M$ is induced from $[a,b]$ in $D^b(k\calP)$.
\item The minimal injective resolution $I_N$ of $N$ over $k\calP$ only has terms $I_x$ with $x\in[a,b]$, and is coinduced from $[a,b]$.
\item The adjunction unit $N\to N\downarrow_{[a,b]}\Uparrow^{\calP}$ is an isomorphism, so that $N$ is coinduced from $[a,b]$ in $D^b(k\calP)$.
\item Induction and coinduction of complexes with homology supported on $(a,b)$ are naturally isomorphic.
\end{enumerate}
\end{proposition}

Note that for any convex subposet $\calQ$ of $\calP$ it is always the case that the unit $L\to L\uparrow^\calP\downarrow_\calQ $ and the counit $L\Uparrow^\calP\downarrow_ \calQ\to L$ are isomorphisms.

\begin{proof}
For the description and properties of projective $k\calP$-modules that we use here see \cite{Web1}, especially Proposition 4.4 in that exposition.
Let $\Pi_M^{[a,b]} \to M\downarrow_{[a,b]}^\calP$ be a minimal projective resolution of $M\downarrow_{[a,b]}^\calP$ over $k[a,b]$. The projective modules that appear in $\Pi_M^{[a,b]}$ have the form $P_x^{[a,b]}$ with $x\in [a,b]$. In the derived category,
$$
M\downarrow_{[a,b]}^\calP\uparrow_{[a,b]}^\calP\cong \Pi_M^{[a,b]}\uparrow_{[a,b]}^\calP=k\calP\otimes_{k[a,b]}\Pi_M^{[a,b]}
$$
is isomorphic to a complex of projectives with summands of the form $P_x^\calP$  with $x\in [a,b]$, as representations of $\calP$, because $P_x^{[a,b]}\uparrow_{[a,b]}^\calP\cong P_x^\calP$.
Furthermore $P_x^{[a,b]}(z)\cong P_x^\calP(z)$ for all $z\in{[a,b]}$. Hence on evaluation at $b$ we have that  $\Pi_M^{[a,b]}\uparrow_{[a,b]}^\calP(b) = \Pi_M^{[a,b]}(b)$ is acyclic, because $H_*(M)$ is zero at $b$. Also, $\Pi_M^{[a,b]}\uparrow_{[a,b]}^\calP$ is only non-zero on elements $y\ge a$, in which case either $y\le b$, in which case  $\Pi_M^{[a,b]}\uparrow_{[a,b]}^\calP(y) = \Pi_M^{[a,b]}(y)$,or $y\ge b$ because $[a,b]$ is clamped. In the latter case the unique morphism $b\to y$ induces an isomorphism
$$
\Pi_M^{[a,b]}\uparrow_{[a,b]}^\calP(b)\to \Pi_M^{[a,b]}\uparrow_{[a,b]}^\calP(y),
$$
because the same is true for each $P_x^\calP$ with $x\in[a,b]$, and thus
$$
\Pi_M^{[a,b]}\uparrow_{[a,b]}^\calP(y)
$$
is acyclic.
We conclude from all this that $\Pi_M^{[a,b]}\uparrow_{[a,b]}^\calP$ is a projective resolution of $M$ over $k\calP$, and the counit of the adjunction
$$
\Pi_M^{[a,b]}\uparrow_{[a,b]}^\calP\to M
$$
provides the isomorphism. The resolution is minimal because $\Pi_M^{[a,b]}$ is. This proves (1) and (2). The proof of (3) and (4) is dual, and (5) follows by combining these results.
\end{proof}

We examine the effect of clamping on Auslander-Reiten triangles in the derived category. Since the incidence algebras of finite posets have finite global dimension, by \cite[IX.10.3]{Mit1} for example, 
it follows that Auslander-Reiten triangles always exist in the bounded derived category, by \cite{Hap}, and they are constructed using the Serre functor $\nu$ on $D^b(k\calP)$ which is the left derived functor of the Nakayama functor, also written $\nu$.

\begin{proposition} \label{clamping-proposition}
Let the interval $[a,b]$ be clamped in the poset $\calP$ and let $M$ be a complex of $k\calP$-modules. Let $\nu$ be the Serre functor on $D^b(k\calP)$. Then  $M$ has homology supported on $[a,b)$ if and only if  $\nu M$ has homology supported on $(a,b]$.  Furthermore, if this is the case and $M$ is assumed to have homology supported on $[a,b)$, then $(\nu M)\downarrow_{[a,b]}^\calP \cong \nu (M\downarrow_{[a,b]}^\calP)$. Dually, if $N$ is a complex whose homology is supported on $(a,b]$, then $(\nu^{-1} N)\downarrow_{[a,b]}^\calP \cong \nu^{-1} (N\downarrow_{[a,b]}^\calP)$.
\end{proposition}

\begin{proof} Suppose that $M$ has non-zero homology only at elements of $[a,b)$. The Serre functor $\nu$ is obtained by applying the Nakayama functor term by term to the projective resolution $\Pi_M$ of $M$, replacing each indecomposable projective by the indecomposable injective with the same socle. Assume that the resolution is minimal, so that by Proposition~\ref{clamping-adjoint-properties} the projective modules that appear have the form $P_x$ with $x\in[a,b]$. We obtain the isomorphism $(\nu \Pi_M)(a) \cong \Pi_M(b)$ because these projectives have the property $(\nu P_x)(a)=I_x(a)\cong k\cong P_x(b)$,
and the maps that appear in the two complexes are the same when transported via this isomorphism. Furthermore, all indecomposable injectives that appear in $\nu \Pi_M$ have the form $I_x$ with $x\in[a,b]$. Since $\Pi_M(b)$ is acyclic, so is $\nu(\Pi_M)(a)$ and hence so is $\nu(\Pi_M)(y)$ if $y\le a$, because the morphism $y\to a$ gives an isomorphism $\nu(\Pi_M)(y)\to \nu(\Pi_M)(a)$. This shows that $\nu(\Pi_M)$ has non-zero homology only at elements of $(a,b]$. The converse statement, assuming that $\nu(\Pi_M)$ has non-zero homology only at elements of $(a,b]$, is proved dually.

By Proposition~\ref{clamping-adjoint-properties} the projectives that appear in the minimal projective resolution of $M$ as a $k\calP$-module have exactly the same labels as those that appear in the minimal projective resolution over $k[a,b]$, because $\Pi_M$ is induced, and so applying the Serre functor commutes with restriction. The dual argument for $N$ is similar.
\end{proof}

We write Auslander-Reiten triangles in the bounded derived category as $\tau N\to M \to N\to \tau N[1]$, where $\tau=\nu[-1]$ is the Auslander-Reiten translate.

\begin{corollary}
Let the interval $[a,b]$ be clamped in the poset $\calP$ and let $M$ be a complex of $k\calP$-modules with homology supported on $[a,b)$. Let $\nu$ be the Serre functor on $D^b(k\calP)$.
\begin{enumerate}
\item For any $N\in D^b(k\calP)$ we have 
$$
\Hom_{D^b(k\calP)}(M,N)\cong \Hom_{D^b(k[a,b])}(M\downarrow_{[a,b]}^\calP,N\downarrow_{[a,b]}^\calP)
$$
\item There is a ring isomorphism
$$
\End_{D^b(k\calP)}(M)\cong \End_{D^b(k[a,b])}(M \downarrow_{[a,b]}^\calP).$$
Identifying ring actions via this isomorphism we have
$$
\Hom_{D^b(k\calP)}(M,\nu M)\cong \Hom_{D^b(k[a,b])}(M \downarrow_{[a,b]}^\calP,\nu M \downarrow_{[a,b]}^\calP)
$$
as $\End_{D^b(k\calP)}(M)$-modules. 
\item On $M$, the Auslander-Reiten translate commutes with restriction: $\tau(M\downarrow_{[a,b]}^\calP)\cong (\tau M)\downarrow_{[a,b]}^\calP$.
\item\label{prop:clamped_intervals_triangles} The Auslander-Reiten triangle in $D^b(k[a,b])$ terminating at $M\downarrow_{[a,b]}^\calP$ is the restriction of the Auslander-Reiten triangle in $D^b(k\calP)$ terminating at $M$, and it is also the restriction of the Auslander-Reiten triangle in $D^b(k\calP)$ starting at $\tau M$.
\end{enumerate}
\end{corollary}

\begin{proof}
(1) We have
$$
\begin{aligned}
\Hom_{D^b(k\calP)}(M,N)&\cong
\Hom_{D^b(k\calP)}(M\downarrow_{[a,b]}^\calP\uparrow_{[a,b]}^\calP,N)\\
&\cong
\Hom_{D^b(k[a,b])}(M\downarrow_{[a,b]}^\calP,N\downarrow_{[a,b]}^\calP)\\
\end{aligned}
$$
by the adjoint property.

(2) These both follow from (1).

(3) This follows because $\tau(U)= \nu(U)[-1]$ for any object $U$, and both $\nu$ and the shift commute with restriction on $M$.

(4) The Auslander-Reiten triangles terminating at $M$ and $M\downarrow_{[a,b]}^\calP$ are related by restriction because they are both computed by taking the mapping cone of a homomorphism $M\to \nu M$ in the first case and  $M\downarrow_{[a,b]}^\calP\to \nu M\downarrow_{[a,b]}^\calP$ in the second case. This homomorphism must lie in the socle for the action of $\End(M)$, and such a homomorphism for $M$ restricts to one for $M\downarrow_{[a,b]}^\calP$ by (2). Since the mapping cone construction also commutes with restriction we obtain that the Auslander-Reiten triangle ending at $M$ restricts as claimed.
\end{proof}

We now put together these technicalities. The next corollary describes the basic relationship between the Auslander-Reiten quivers of the bounded derived categories of a poset and of a clamped subinterval, showing that part of one is copied into the other.

\begin{corollary}\label{AR-quiver-copying-corollary}
Let the interval $[a,b]$ be clamped in the poset $\calP$ and let $\Theta$ be the region of the Auslander-Reiten quiver of $D^b(k\calP)$ consisting of the meshes where either the right hand term has homology supported on  $[a,b)$ or the left hand term has homology supported on  $(a,b]$. Then all complexes in $\Theta$ have homology supported on $[a,b]$ and restriction $\downarrow_{[a,b]}^\calP$ gives an isomorphism between $\Theta$ and the corresponding region of the Auslander-Reiten quiver of $D^b([a,b])$.
\end{corollary}

\begin{example}\label{A4-example}
We illustrate Corollary~\ref{AR-quiver-copying-corollary} with a straightforward example. We take the poset $\calP$ to be the chain
\begin{center}
\begin{tikzpicture}[xscale=.5,yscale=.5]
\draw (0,0)--(0,1)--(0,2)--(0,3);
\draw[fill] (0,0) circle [radius=0.08];
\draw[fill] (0,1) circle [radius=0.08];
\draw[fill] (0,2) circle [radius=0.08];
\draw[fill] (0,3) circle [radius=0.08];

\node [ right] at  (0,0)  {$4$};
\node [ right] at  (0,1)  {$3$};
\node [ right] at  (0,2)  {$2$};
\node [ right] at  (0,3)  {$1$};
\node at (-3,1.5) {$\calP=$};
\end{tikzpicture}
\end{center}
whose representations are the same as those of the Dynkin quiver $A_4$ with all arrows oriented in the same direction. The Auslander-Reiten quiver of $D^b(k\calP)$ is familiar from~\cite{Hap} and part of it looks as follows.
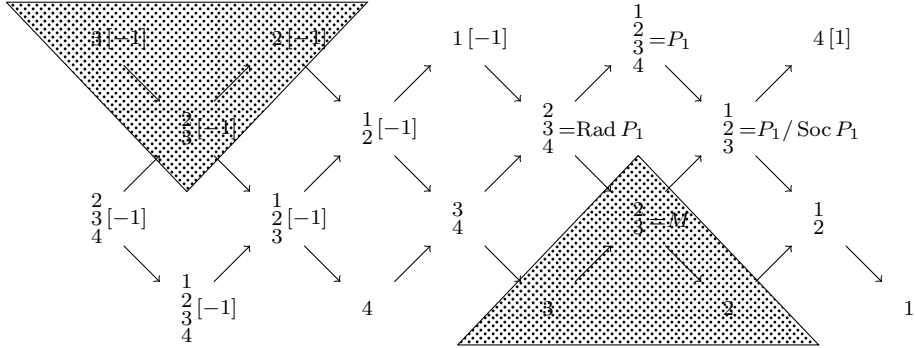
\begin{center}
\begin{tikzpicture}[xscale=1.2,yscale=1.2]

\draw[->] (-1.7,0.3)--(-1.3,0.7);
\draw[->] (0.3,0.3)--(0.7,0.7);
\draw[->] (2.3,0.3)--(2.7,0.7);
\draw[->] (4.3,0.3)--(4.7,0.7); 

\draw[->] (-2.7,1.3)--(-2.3,1.7);
\draw[->] (-0.7,1.3)--(-0.3,1.7);
\draw[->] (1.3,1.3)--(1.7,1.7);
\draw[->] (3.3,1.3)--(3.7,1.7); 

\draw[->] (-1.7,2.3)--(-1.3,2.7);
\draw[->] (0.3,2.3)--(0.7,2.7);
\draw[->] (2.3,2.3)--(2.7,2.7);
\draw[->] (4.3,2.3)--(4.7,2.7);

\draw[->] (-2.7,0.7)--(-2.3,0.3);
\draw[->] (-0.7,0.7)--(-0.3,0.3);
\draw[->] (1.3,0.7)--(1.7,0.3);
\draw[->] (3.3,0.7)--(3.7,0.3);
\draw[->] (5.3,0.7)--(5.7,0.3);

\draw[->] (-1.7,1.7)--(-1.3,1.3);
\draw[->] (0.3,1.7)--(0.7,1.3);
\draw[->] (2.3,1.7)--(2.7,1.3);
\draw[->] (4.3,1.7)--(4.7,1.3);

\draw[->] (-2.7,2.7)--(-2.3,2.3);
\draw[->] (-0.7,2.7)--(-0.3,2.3);
\draw[->] (1.3,2.7)--(1.7,2.3);
\draw[->] (3.3,2.7)--(3.7,2.3);

\node at (-2,0) {$\scriptstyle{\begin{smallmatrix}1\\ 2\\ 3 \\ 4 \end{smallmatrix}}\rlap{$\scriptstyle [-1]$}$};
\node at (0,0) {${\begin{smallmatrix}4 \end{smallmatrix}}$};
\node at (2,0) {${\begin{smallmatrix}3 \end{smallmatrix}}$};
\node at (4,0) {${\begin{smallmatrix}2 \end{smallmatrix}}$};
\node at (6,0) {${\begin{smallmatrix}1 \end{smallmatrix}}$};

\node at (-3,1) {$\scriptstyle{\begin{smallmatrix}2\\ 3 \\ 4 \end{smallmatrix}}\rlap{$\scriptstyle [-1]$}$};
\node at (-1,1) {$\scriptstyle{\begin{smallmatrix}1\\ 2\\ 3 \end{smallmatrix}}\rlap{$\scriptstyle [-1]$}$};
\node at (1,1) {${\begin{smallmatrix}3 \\ 4 \end{smallmatrix}}$};
\node at (3,1) {${\begin{smallmatrix}2\\ 3 \end{smallmatrix}\rlap{$\scriptstyle =M$}}$};
\node at (5,1) {${\begin{smallmatrix}1 \\ 2 \end{smallmatrix}}$};

\node at (-2,2) {$\scriptstyle{\begin{smallmatrix} 2\\ 3 \end{smallmatrix}}\rlap{$\scriptstyle [-1]$}$};
\node at (0,2) {$\scriptstyle{\begin{smallmatrix} 1\\ 2 \end{smallmatrix}}\rlap{$\scriptstyle [-1]$}$};
\node at (2,2) {${\begin{smallmatrix} 2\\ 3 \\ 4\end{smallmatrix}\rlap{$\scriptstyle=\Rad P_1$}}$};
\node at (4,2) {${\begin{smallmatrix}1 \\ 2\\ 3 \end{smallmatrix}\rlap{$\scriptstyle=P_1/\Soc P_1$}}$};

\node at (-3,3) {$\scriptstyle{\begin{smallmatrix} 3 \end{smallmatrix}}\rlap{$\scriptstyle [-1]$}$};
\node at (-1,3) {$\scriptstyle{\begin{smallmatrix} 2 \end{smallmatrix}}\rlap{$\scriptstyle [-1]$}$};
\node at (1,3) {$\scriptstyle{\begin{smallmatrix} 1 \end{smallmatrix}}\rlap{$\scriptstyle [-1]$}$};
\node at (3,3) {${\begin{smallmatrix} 1\\ 2\\ 3\\ 4 \end{smallmatrix}\rlap{$\scriptstyle=P_1$}}$};
\node at (5,3) {$\scriptstyle{\begin{smallmatrix} 4 \end{smallmatrix}}\rlap{$\scriptstyle [1]$}$};

\draw[pattern=crosshatch dots] (1, -0.4)--(5,-0.4)--(3,1.7)--(1,-0.4);
\draw[pattern=crosshatch dots] (-4, 3.4)--(0,3.4)--(-2,1.3)--(-4, 3.4);

\end{tikzpicture}
\end{center}
We consider the clamped interval $[2,3]$. The region $\Theta$ described in Corollary~\ref{AR-quiver-copying-corollary} consists of infinitely many triangles, of which two are shown shaded. These triangles make up the quiver of $D^b(k[2,3])$ and are copied into the quiver of $D^b(k\calP)$. We will see in Proposition~\ref{glue-proposition} exactly how these triangles are glued in to the bigger quiver. For future reference we have also labeled terms in the Auslander-Reiten triangle containing the indecomposable projective $P_1$.
\end{example}

In the next sections we will use this isomorphism between parts of the Auslander-Reiten quiver for a poset and a clamped interval in two different ways. Our first application will be to give a criterion in the next section for the bounded derived category of a poset to fail to be fractionally Calabi-Yau. After that we will show how this approach can be used to construct the Auslander-Reiten quivers of both the bounded derived categories and of the module categories of posets.

\section{Posets that are fractionally Calabi-Yau}
We say that a triangulated category is \textit{Calabi-Yau} if some power of the shift functor is a Serre functor $\nu$. More generally, we say that the category is \textit{fractionally Calabi-Yau} if it has a Serre functor $\nu$ and $\nu^n=[m]$ for some powers $m,n$ of the shift functor and the Serre functor. 
Algebras whose bounded derived categories are fractionally Calabi-Yau have received attention recently in the work of several authors, and we mention \cite{Cha}, \cite{HI}, \cite{KLM} and \cite{Sch} as examples.

In the context of the bounded derived category of a finite dimensional algebra $\Lambda$ the Auslander-Reiten translate on $D^b(\Lambda)$ has the form $\tau(M)=\nu(M)[-1]$ so that, when Auslander-Reiten triangles exist, $D^b(\Lambda)$ is fractionally Calabi-Yau if and only if the Auslander-Reiten quiver is periodic up to some power of the shift.
This possibility was considered by Scherotzke in \cite[Theorem 4.13]{Sch}, and she showed that if a stable Auslander-Reiten component of $D^b(\Lambda)$ contains a shift-periodic complex, then its tree class is either a finite Dynkin diagram or $A_{\infty}$. 
Furthermore, in the Dynkin case this is the only Auslander-Reiten component of $D^b(\Lambda)$, and $D^b(\Lambda)$ has `finite representation type'. This means that there are only finitely many indecomposable objects in $D^b(\Lambda)$ up to shift, and it implies that $\Lambda$ itself has finite representation type.  It is  interesting to understand which algebras of infinite representation type have the fractional Calabi-Yau property. 

Using Scherotzke's result, we present the following as a corollary of Corollary~\ref{AR-quiver-copying-corollary}.

\begin{corollary} \label{cor:inf_rep_type} Let $\calP$ be a poset for which  $k\calP$ has infinite representation type and which contains a clamped subinterval $\calQ =[a,b]$. Suppose there is an object of $D^b(k\calQ)$ represented by a complex $M$ with homology supported only in $[a,b)$ such that the Auslander-Reiten triangle terminating in $M$ has more than two middle terms. Then $D^b(k\calP)$  is not fractionally Calabi-Yau.
\end{corollary}

The same conclusion holds if there is a complex $M$ with homology supported only in $(a,b]$ if  the Auslander-Reiten triangle starting in $M$ has more than two middle terms.

\begin{proof} Suppose to the contrary that  $D^b(k\calP)$  is fractionally Calabi-Yau and consider the Auslander-Reiten component of $D^b(k\calP)$ containing $M$. By Scherotzke's theorem \cite[4.13]{Sch} we conclude that the tree class of the component must be $A_\infty$, since $k\calP$ has infinite representation type. Take $\calQ = [a,b]$ in Proposition~\ref{clamping-proposition}. Since the Auslander-Reiten triangle terminating in $M$ has more than two indecomposable middle terms, part   (\ref{prop:clamped_intervals_triangles}) of Proposition~\ref{clamping-proposition}  says that $D^b(k\calP)$ also has an Auslander-Reiten triangle with more than two indecomposable middle terms and so the tree class cannot by $A_\infty$. This contradiction allows us to conclude that $D^b(k\calP)$  is not fractionally Calabi-Yau.
\end{proof}

\begin{example}\label{chain-product-example}
Let $\calP$ be a poset for which $k\calP$ has infinite representation type. Suppose that one of the following three intervals $\calQ=[a,b]$ is clamped in $\calP$:
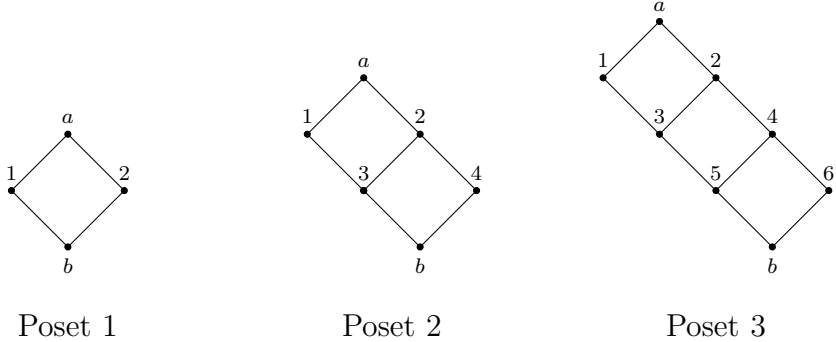
\begin{center}
\begin{minipage}[b]{0.3\linewidth}
\begin{tikzpicture}[xscale=.5,yscale=.5]
\draw (2,0)--(0.5,1.5)--(2,3)--(3.5,1.5)--(2,0);
\draw[fill] (2,0) circle [radius=0.08];
\draw[fill] (0.5,1.5) circle [radius=0.08];
\draw[fill] (2,3) circle [radius=0.08];
\draw[fill] (3.5,1.5) circle [radius=0.08];

\node [ below] at  (2,0)  {\tiny $b$};
\node [ above] at  (2,3)   {\tiny $a$};
\node [ above] at  (.5,1.5)   {\tiny $1$};
\node [ above] at  (3.5,1.5)   {\tiny $2$};

\node [ below] at  (2,-1.5)   {Poset 1};
\end{tikzpicture}
\end{minipage}
\begin{minipage}[b]{0.3\linewidth}
\begin{tikzpicture}[xscale=.5,yscale=.5]
\draw (4,0)--(2.5,1.5)--(4,3)--(5.5,1.5)--(4,0);
\draw[fill] (4,0) circle [radius=0.08];
\draw[fill] (2.5,1.5) circle [radius=0.08];
\draw[fill] (2.5,4.5) circle [radius=0.08];
\draw[fill] (4,3) circle [radius=0.08];
\draw[fill] (5.5,1.5) circle [radius=0.08];

\draw (4,3)--(2.5,4.5)--(1,3)--(2.5,1.5);
\draw[fill] (1,3) circle [radius=0.08];
\draw[fill] (2.5,1.5) circle [radius=0.08];

\node [ below] at  (4,0)  {\tiny $b$};
\node [ above] at  (2.5,4.5)   {\tiny $a$};
\node [ above] at  (1,3)   {\tiny $1$};
\node [ above] at  (4,3)   {\tiny $2$};
\node [ above] at  (2.5,1.5)   {\tiny $3$};
\node [ above] at  (5.5,1.5)   {\tiny $4$};

\node [ below] at  (3.25,-1.5)   {Poset 2};
\end{tikzpicture}
\end{minipage}
\begin{minipage}[b]{0.3\linewidth}
\begin{tikzpicture}[xscale=.5,yscale=.5]
\draw (5,0)--(3.5,1.5)--(5,3)--(6.5,1.5)--(5,0);
\draw[fill] (5,0) circle [radius=0.08];
\draw[fill] (3.5,1.5) circle [radius=0.08];
\draw[fill] (5,3) circle [radius=0.08];
\draw[fill] (6.5,1.5) circle [radius=0.08];
\draw (5,3)--(3.5,4.5)--(2,3)--(3.5,1.5);
\draw[fill] (2,3) circle [radius=0.08];
\draw[fill] (3.5,4.5) circle [radius=0.08];
\draw (2,3)--(0.5,4.5)--(2,6)--(3.5,4.5);
\draw[fill] (0.5,4.5) circle [radius=0.08];
\draw[fill] (2,6) circle [radius=0.08];

\node [ below] at  (5,0)  {\tiny $b$};
\node [ above] at  (2,6)   {\tiny $a$};
\node [ above] at  (3.5,1.5)   {\tiny $5$};
\node [ above] at  (6.5,1.5)   {\tiny $6$};
\node [ above] at  (2,3)   {\tiny $3$};
\node [ above] at  (5,3)   {\tiny $4$};
\node [ above] at  (.5,4.5)   {\tiny $1$};
\node [ above] at  (3.5,4.5)   {\tiny $2$};

\node [ below] at  (3.5,-1.5)   {Poset 3};
\end{tikzpicture}
\end{minipage}
\end{center}

We will show in each case that $D^b(k\calP)$ is not fractionally Calabi-Yau. We apply Corollary \ref{cor:inf_rep_type} using a certain $k\calP$-complex $M$ with homology supported in  $[a,b)$, and in each case the complex will be a module in degree 0. As before, let $P_x$ denote the indecomposable projective representation of $k\calP$ generated at vertex $x$. When $\calP$ contains the clamped interval described by the first diagram, take $M=P_a/P_b$. For the second diagram take $M=\Rad P_a/P_b$. In the third case we take $M$ to be the indecomposable middle term of the Auslander-Reiten triangle terminating in the indecomposable complex $P_6 \to P_3$. Then $M$ is a $k\calP$-module of the form $\begin{smallmatrix} &0&&& \\ k && k^2  \\ & k^2&&k^2 \\&&k&& k \\ &&&0 \end{smallmatrix}$ which we will denote as $ \left[ \begin{smallmatrix} 1 &0 \\ 2&2 \\ 1&2 \\ 0 &1 \end{smallmatrix}\right]$.
By Proposition \ref{clamping-proposition}, and in each case, the Auslander-Reiten triangles over $k\calP$ and $k\calQ$ terminating at $M$ are the same. Furthermore, one calculates over $k\calQ$ that in each case the Auslander-Reiten triangle terminating at $M$  has three indecomposable middle terms. In particular the following are Auslander-Reiten triangles:

\[ \tag{Poset 1} \Rad P_a \to P_a\oplus k_1 \oplus k_2 \to P_a/P_b\to \Rad P_a[1]
\]

\[ 
\tag{Poset 2}  
\left( \begin{smallmatrix}   P_b\\
\downarrow \\ P_1 \oplus P_2  \end{smallmatrix} \right) \to  
\Rad P_a\oplus P_1/P_b \oplus P_2/P_b\to(\Rad P_a)/P_b     \to 
\left( \begin{smallmatrix}  P_b\\
\downarrow \\
P_1 \oplus P_2  \end{smallmatrix} \right)  [1].
\]
In the above triangle the two middle terms are modules in degree 0, and the two end complexes (which are also isomorphic to shifts of modules) have their lowest shown terms in degree 0.

For the third poset, the calculations become more involved but we again find an Auslander-Reiten triangle consisting of modules in degree 0. We will use dimension vectors of representations to simplify the notation. The following describes the first three terms of an Auslander-Reiten triangle in $D^b(k\calQ)$ with three indecomposable middle summands:
\[ 
\leqno{\hbox{(Poset 3) }}
\left[ \begin{smallmatrix} 1 &0 \\ 2&1 \\ 1&2 \\ 0 &1 \end{smallmatrix}\right]
\to \left[ \begin{smallmatrix} 1 &0 \\ 1&1 \\ 1&2 \\ 0 &1 \end{smallmatrix}\right]
  \oplus \left[ \begin{smallmatrix} 0 &0 \\ 1&1 \\ 0&1 \\ 0 &1 \end{smallmatrix}\right]
\oplus \left[ \begin{smallmatrix} 1 &0 \\ 2&1 \\ 1&1 \\ 0&0 \end{smallmatrix}\right]
 \to \left[ \begin{smallmatrix} 1 &0 \\ 2&2 \\ 1&2 \\ 0 &1 \end{smallmatrix}\right]\to \left[ \begin{smallmatrix} 1&0 \\ 2&1 \\ 1&2 \\ 0&1 \end{smallmatrix}\right] [1] .\]

In fact, for each of the three posets in our example the Auslander-Reiten quiver of $D^b(k\calQ)$ is a single component of Dynkin type $D_4$, $E_6$, and $E_8$, respectively. An application of Corollary \ref{cor:inf_rep_type} then shows that $k\calP$ is not fractionally Calabi-Yau.

It is interesting to note that continuing this pattern and taking four stacked boxes for the poset $\calQ$, does \emph{not} yield a restriction on whether $D^b (k\calP)$ is fractionally Calabi-Yau. This is because the Auslander-Reiten triangles for $k\calQ$ have at most two middle terms.
\end{example}

\begin{example}\label{canonical-algebra-example} As another example, we use clamping theory to show that incidence algebras of posets $\calP$ that have the following posets $\calQ$  as clamped intervals are not fractionally Calabi-Yau.


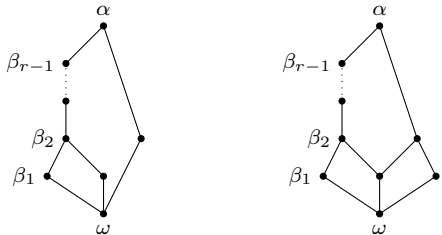
\begin{center}
\begin{minipage}[b]{0.3\linewidth}
\centering
\begin{tikzpicture}[xscale=.5,yscale=.5]
\draw (2,0)--(0.5,1)--(1,2)--(1,3);
\draw (2,5)--(3,2)--(2,0);
\draw (1,4)--(2,5);
\draw (2,0)--(2,1)--(1,2);
\draw[dotted] (1,3)--(1,4) ;
\draw[fill] (2,0) circle [radius=0.08];
\draw[fill] (0.5,1) circle [radius=0.08];
\draw[fill] (2,1) circle [radius=0.08];
\draw[fill] (1,2) circle [radius=0.08];
\draw[fill] (2,5) circle [radius=0.08];
\draw[fill] (1,3) circle [radius=0.08];
\draw[fill] (1,4) circle [radius=0.08];
\draw[fill] (3,2) circle [radius=0.08];

\node [ below] at  (2,0)  {\tiny $\omega$};
\node [ left] at  (0.5,1)   {\tiny $\beta_1$};
\node [ left] at  (1,2)   {\tiny $\beta_2$};
\node [ left] at  (1,4)   {\tiny $\beta_{r-1}$};
\node [ above] at  (2,5)   {\tiny $\alpha$};
\end{tikzpicture}
\end{minipage}%
\begin{minipage}[b]{0.3\linewidth}
\centering
\begin{tikzpicture}[xscale=.5,yscale=.5]
\draw (3,2)--(3.5,1)-- (2,0)--(0.5,1)--(1,2)--(1,3);
\draw (2,5)--(3,2)--(2,1);
\draw (1,4)--(2,5);
\draw (2,0)--(2,1)--(1,2);
\draw[dotted] (1,3)--(1,4) ;
\draw[fill] (3.5,1) circle [radius=0.08];
\draw[fill] (2,0) circle [radius=0.08];
\draw[fill] (0.5,1) circle [radius=0.08];
\draw[fill] (2,1) circle [radius=0.08];
\draw[fill] (1,2) circle [radius=0.08];
\draw[fill] (2,5) circle [radius=0.08];
\draw[fill] (1,3) circle [radius=0.08];
\draw[fill] (1,4) circle [radius=0.08];
\draw[fill] (3,2) circle [radius=0.08];

\node [ below] at  (2,0)  {\tiny $\omega$};
\node [ left] at  (0.5,1)   {\tiny $\beta_1$};
\node [ left] at  (1,2)   {\tiny $\beta_2$};
\node [ left] at  (1,4)   {\tiny $\beta_{r-1}$};
\node [ above] at  (2,5)   {\tiny $\alpha$};
\end{tikzpicture}
\end{minipage}
\end{center}

For the posets $\calQ$ corresponding to these diagrams, the incidence algebra $k\calQ$ is derived equivalent to a canonical algebra of Euclidean type   \cite[1.1]{Lad}. Canonical algebras of Euclidean type are a class of finite-dimensional algebras of global dimension at most $2$ that have an Auslander-Reiten quiver component whose section is a Euclidean diagram  \cite{SS}. They are of special interest as they are of minimal infinite representation type in the sense of Happel-Vossieck \cite{HV}. More precisely, the posets $\calQ$ corresponding to the diagram on the left have incidence algebras that are derived equivalent to canonical algebras that are denoted in the literature as  $C(2,2,r)$ where $2\leq r$. These are said to be of Euclidean type $\widetilde D_{r+2}$.   The posets corresponding to the diagram on the right on the other hand have incidence algebras that are derived equivalent to canonical algebras  $C(2,3,r)$ where $r=3, 4, 5$. These are of Euclidean type $\widetilde E_{i}$ for $i=6, 7, 8$ respectively.  We refer the reader to \cite{SS} for further background on these algebras.

Whenever one of the above posets $\calQ$ is clamped in a poset $\calP$, an application of Corollary \ref{cor:inf_rep_type} shows that the incidence algebra $k\calP$ is not fractionally Calabi-Yau.
\end{example}

\section{Iterated clamping}

In this section and the next we turn our attention to posets constructed by repeatedly clamping intervals in an inductive process. It turns out that the posets constructed in this way are all piecewise hereditary, meaning that they are derived equivalent to hereditary algebras. We are able to specify combinatorially the tree of the corresponding hereditary algebra and hence determine the entire  Auslander-Reiten quiver of $D^b(k\calP)$. We are  also able to reverse-engineer this process and construct posets that are piecewise hereditary with a given tree class for many trees, including all trees that are barycentric subdivisions. 

To this end we define a class of posets $\calI\calC$ recursively as follows. All of the posets will have a unique minimal element $\alpha$ and a unique maximal element $\omega$. We define $\calI\calC_0$ to contain only the poset with a single element: $\bullet\;\alpha=\omega$. Now supposing that $n\ge 1$ and that $\calI\calC_{n-1}$ has been defined, we define $\calI\calC_n$ to consist of the posets that have a unique minimal element $\alpha$ and a unique maximal element $\omega$ so that the open interval $(\alpha,\omega)$ is a disjoint union of posets from $\calI\calC_{n-1}$:
$$
(\alpha,\omega)=P_1\sqcup\cdots\sqcup P_s,\quad s\ge0, \quad P_i\in\calI\calC_{n-1}.
$$
The interpretation of $s=0$ is that the disjoint union could be empty, in which case the poset consists either of a single point, or of two points that are comparable. We easily see that $\calI\calC_{n-1}\subset\calI\calC_n$ always, arguing by induction and starting with $\calI\calC_0\subset\calI\calC_1$. We define
$$
\calI\calC=\bigcup_{i=0}^\infty \calI\calC_n.
$$
The posets in $\calI\calC$ are obtained by Iterative Clamping.
Thus, for example, posets such as 
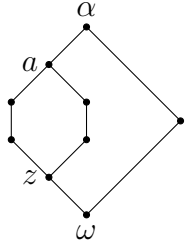
\begin{center}
\begin{tikzpicture}[xscale=.5,yscale=.5]
\draw (2,0)--(1,1)--(0,2)--(0,3)--(1,4)--(2,5)--(4.5,2.5)--(2,0);
\draw (1,1)--(2,2)--(2,3)--(1,4);
\draw[fill] (2,0) circle [radius=0.08];
\draw[fill] (1,1) circle [radius=0.08];
\draw[fill] (0,2) circle [radius=0.08];
\draw[fill] (0,3) circle [radius=0.08];
\draw[fill] (1,4) circle [radius=0.08];
\draw[fill] (2,5) circle [radius=0.08];
\draw[fill] (4.5,2.5) circle [radius=0.08];
\draw[fill] (2,2) circle [radius=0.08];
\draw[fill] (2,3) circle [radius=0.08];
\node [ below] at  (2,0)  {$\omega$};
\node [ left] at  (1,1)   {$z$};
\node [ left] at  (1,4)   {$a$};
\node [ above] at  (2,5)   {$\alpha$};
\end{tikzpicture}
\end{center}
lie in $\calI\calC$ but posets 2 and 3 in Example~\ref{chain-product-example} 
and the posets in Example~\ref{canonical-algebra-example} do not. The labels in this picture are relevant for Proposition~\ref{glue-proposition}. We observe that the posets in $\calI\calC$ are all lattices.

The next result applies to posets in $\calI\calC$, and also more generally. It identifies the form of certain Auslander-Reiten triangles and shows how the part of the Auslander-Reiten quiver of $D^b(k\calP)$ identified in Corollary~\ref{AR-quiver-copying-corollary} is glued on to the rest of the quiver. In our notation we will identify a $k\calP$-module $M$ with the complex of $k\calP$-modules whose only term is $M$, concentrated in degree 0. An example illustrating this glueing process was given as Example~\ref{A4-example}.

\begin{proposition}
\label{glue-proposition}
Let $\calP$ be a poset with a unique minimal element $\alpha$ and a unique maximal element $\omega$, with $\alpha\ne\omega$.
\begin{enumerate}
\item There are Auslander-Reiten triangles
$$
I_\alpha[-1]\to\Rad P_\alpha \to P_\alpha\to I_\alpha,
$$
$$
P_\alpha\to P_\alpha/\Soc P_\alpha \to \Soc P_\alpha[1]\to P_\alpha [1]\
$$
and
$$
\Rad P_\alpha\to P_\alpha\oplus (\Rad P_\alpha/\Soc P_\alpha)\to P_\alpha/\Soc P_\alpha\to \Rad P_\alpha[1].
$$
\item Suppose, further, that there is a clamped interval $[a,z]$ in $\calP$ where $a$ covers $\alpha$ and $\omega$ covers $z$. Let $M=P_a/P_\omega=k_{[a,z]}$ in the notation used previously. Then there are Auslander-Reiten triangles
$$
k_{\{\calP-\{\alpha,a\}\}}\to\Rad P_\alpha\oplus\Rad M\to M\to k_{\{\calP-\{\alpha,a\}\}}[1]
$$
and
$$
M\to P_\alpha/\Soc(P_\alpha)\oplus M/\Soc(M)\to k_{\{\calP-\{\omega,z\}\}}\to M[1].
$$
\end{enumerate}
\end{proposition} 

\begin{proof} (1) We establish that these are Auslander-Reiten triangles by direct calculation. The AR triangle ending at $P_\alpha$ is obtained by completing to a triangle the map $P_\alpha\to I_\alpha$ that sends the simple top to the simple socle, and rotating. The mapping cone of this map is isomorphic to $\Rad P_\alpha [1]$. The AR triangle starting at $P_\alpha$ is obtained similarly, noting that $P_\alpha=I_\omega$. Note also that $\Soc P_\alpha=P_\omega$.

For the second triangle, the module $P_\alpha/\Soc P_\alpha$ is isomorphic to a complex $P_\omega\to P_\alpha$, and applying the Nakayama functor we get $I_\omega\to I_\alpha\cong(\Rad P_\alpha)[1]$. 
Thus there is an Auslander-Reiten triangle with end terms  $\Rad P_\alpha$ and $P_\alpha/\Soc P_\alpha$ as claimed, and whose middle term is the mapping cone of a map $(P_\omega\to P_\alpha)[-1]\to (I_\omega\to I_\alpha)[-1]$. This mapping cone is a complex $P_\omega\to P_\alpha\oplus I_\omega \to I_\alpha$,  which is seen to be non-zero only in degree 0. Note that $P_\alpha\cong I_\omega$ is both projective and injective and that $P_\omega$ and $I_\alpha$ are both simple. Because of this the mapping cone is the direct sum of a diagonally embedded copy of $P_\alpha\cong\{(x,-x)\bigm|x\in P_\alpha\}$ and a complex $(P_\omega\to P_\alpha\to I_\alpha)\cong \Rad P_\alpha / \Soc P_\alpha$, with both modules in degree 0.

(2) Suppose that there is a clamped interval $[a,z]$ in $\calP$ as in the statement of the proposition and write $N=k_{\{\calP-\{\alpha,a\}\}}$.
Let $M=P_a/P_\omega=k_{[a,z]}$.
We know that $M$ is isomorphic to the complex $P_\omega\to P_a$ so that $\nu(M)= I_\omega\to I_a =N[1]$ is the shift of $N$. We also know that there is an Auslander-Reiten triangle $\Rad P_\alpha\to P_\alpha\oplus M\oplus U\to P_\alpha/\Soc P_\alpha\to \Rad P_\alpha[1]$ for some module $U$ by what we already proved, since $M$ is a summand of $\Rad P_\alpha / \Soc P_\alpha$. Thus there is an irreducible morphism $\Rad P_\alpha\to M$. There is also an irreducible morphism $\Rad M\to M$ because the Auslander-Reiten quiver of $[a,b]$ is copied in the Auslander-Reiten quiver of $\calP$ by Corollary~\ref{AR-quiver-copying-corollary}, and by part (1) of this proposition applied to $[a,b]$.

The mapping cone $C$ of $M[-1]\to \nu M[-1]$ has non-zero homology only in degree 0 by an argument already used in part (1), so it is equivalent to a module, and the long exact sequence in homology of the triangle $\nu M[-1]\to C\to M\to \nu M$ becomes a short exact sequence $0\to N\to H_0(C)\to M\to 0$ which is the start of the Auslander-Reiten triangle with $M$ on the right. We have seen that $\Rad P_\alpha$ and $\Rad M$ are summands of $H_0(C)$. Since
$$
\dim M+\dim N=\dim\Rad P_\alpha+\dim\Rad M
$$
we have $H_0(C)=\Rad P_\alpha\oplus \Rad M$.

The argument for the Auslander-Reiten triangle starting at $M$ is similar.
\end{proof}

By a \textit{slice} of a component of an Auslander-Reiten quiver we mean a maximal connected subgraph with the property that, for each vertex $x$, at most one element of $\{x,\tau x\}$ belongs to the subgraph.
So that we can apply an inductive argument to posets obtained by iterative clamping, we consider posets satisfying the following technical condition:

\begin{hypothesis} 
\label{wing-hypothesis}
The poset $\calP$ has a unique minimal element $a$ and a unique maximal element $z$. Furthermore, the component of the Auslander-Reiten quiver of $D^b(k\calP)$ that contains the largest projective module $P_a$ has a slice that consists of $P_a$ and some other modules (regarded as complexes concentrated in degree 0). Apart from $P_a$, none of these modules have $z$ in their support.
\end{hypothesis}

The next result contains the inductive step that shows that posets in $\calI\calC$ satisfy Hypothesis~\ref{wing-hypothesis}.

\begin{proposition}
\label{inductive-clamping-proposition}
Let $\calP$ be a poset with a unique minimal element $\alpha$ and a unique maximal element $\omega$ such that the open interval
$$
(\alpha,\omega)=[a_1,z_1]\sqcup\cdots\sqcup[a_n,z_n]
$$
is a disjoint union of closed intervals. Suppose that, for each $i$, the Auslander-Reiten quiver component of $D^b(k[a_i,z_i])$ that contains the projective $P_{a_i}$  has a slice $\calS_i$ with tree class $T_i$ satisfying Hypothesis~\ref{wing-hypothesis}. 
\begin{enumerate}
\item \label{inductive-clamping-proposition:step1}
Then $\calP$ also satisfies Hypothesis~\ref{wing-hypothesis} with the slice
$$
\calS= \calS_1\cup \cdots \cup \calS_n\cup \{P_\alpha, P_\alpha/\Soc P_\alpha\}.
$$
\item The tree class $T$ of the Auslander-Reiten quiver component of $D^b(k\calP)$ that contains the largest projective $P_\alpha$ is obtained by gluing the $T_i$ to the ends of a graph that is a star consisting of $n+1$ edges joined at a single point. One of these edges has nothing glued to it, and the remaining edges have the $T_i$ glued to them at the point where the projective $P_{a_i}$ was fastened.
\item If the direct sum of modules in each slice $\calS_i$ is a tilting complex $U_i$ for the interval $[a_i,z_i]$, then the direct sum of modules in the slice $\calS$ is a tilting complex $U$ for $k\calP$. 
\item If each ring $\End_{D^b(k[a_i,z_i])}(U_i)$ is hereditary then so is $\End_{D^b(k\calP)}(U)$, and so $k\calP$ is piecewise hereditary.
\end{enumerate}
\end{proposition}

\begin{proof}
By Proposition~\ref{glue-proposition} there is an Auslander-Reiten triangle in $D^b(k\calP)$ with terms
$$
\Rad P_\alpha\to P_\alpha\oplus k_{[a_1,z_1]}\oplus\cdots\oplus k_{[a_n,z_n]}\to P_{\alpha}/P_\omega\to \Rad P_\alpha[1].
$$
By Hypothesis~\ref{wing-hypothesis}, for each $i$ we have a slice $\calS_i$ that includes the module $k_{[a_i,z_i]}$ and whose remaining terms do not have $z_i$ in their support. By Proposition~\ref{clamping-proposition} the Auslander-Reiten triangles in $D^b(k_{[a_i,z_i]})$ that have these remaining terms on the right remain Auslander-Reiten triangles for $k\calP$. This means that the slice $\calS_i$ remains part of a slice for $k\calP$.
By Proposition~\ref{glue-proposition} it is glued to the rest of the quiver in the manner described in (2), where we include the modules $P_\alpha$ and $P_\alpha/\Soc(P_\alpha)$ in the slice. Again by Proposition~\ref{glue-proposition}, the slice is complete. The only term in this slice that has $\omega$ in its support is $P_\alpha$, so that $\calP$ satisfies Hypothesis~\ref{wing-hypothesis}. This proves parts (1) and (2).

For the proof of (3) we first show that the slice just constructed has the property that, for every pair of modules $M,N$ in it, $\Hom_{D^b(k\calP)}(M,N[i])=0$ if $i\ne 0$. We consider several cases. Suppose first that $M\in\calS_r$ and $N\in\calS_s$ with $M$ supported on $[a_r,z_r)$. Then by Proposition~\ref{clamping-adjoint-properties} $M$ is induced from $k[a_r,z_r]$ and
$$\Hom_{D^b(k\calP)}(M,N[i]) \cong \Hom_{D^b(k[a_r,z_r])}(M\downarrow_{[a_r,z_r]}^\calP,N[i]\downarrow_{[a_r,z_r]}^\calP).
$$
If $r\ne s$ then $N[i]\downarrow_{[a_r,z_r]}=0$, so this Hom group is 0. If $r=s$ then the Hom group is 0 if $i\ne 0$ from the hypothesis on $\calS_r$.

We next consider the case when $M=k_{[a_r,z_r]}$ is the largest projective module for $k[a_r,z_r]$. As an object in $D^b(k\calP)$ it is isomorphic to the complex $P_\omega\to P_{a_r}$. Because this projective resolution has terms only in degrees 0 and 1, if $N$ is any $k\calP$-module then $\Hom_{D^b(k\calP)}(M,N[i])=0$ unless $i$ is 0 or 1, and if $i=1$ and there is a non-zero homomorphism then the simple module $P_\omega$ lies in $\Soc N$. When $N$ lies in the slice under consideration this means $N=P_\alpha$, which is injective as well as projective, so $\Hom_{D^b(k\calP)}(M,P_\alpha[1])=0$ in this case. This shows that for this $M$ we have $\Hom_{D^b(k\calP)}(M,N[i])=0$ when $N$ lies in the slice and $i\ne 0$.

The argument is similar when $M=P_\alpha/\Soc P_\alpha$. This module is isomorphic to the complex $P_\omega\to P_\alpha$ so that for any $k\calP$-module $N$ we have $\Hom_{D^b(k\calP)}(M,N[i])=0$ unless $i$ is 0 or 1, and if $i=1$ and there is a non-zero homomorphism then the simple module $P_\omega$ lies in $\Soc N$. For $N$ in the slice this again implies that $N=P_\alpha$, which is injective, so $\Hom_{D^b(k\calP)}(M,P_\alpha[1])=0$. 

We have now accounted for $\Hom_{D^b(k\calP)}(M,N[i])$ in all cases except when one of $M$ or $N$ is $P_\alpha$. As already observed, this module is both projective and injective, so in these remaining cases the Hom group is 0 if $i\ne 0$.

We next show that $D^b(k\calP)$ is generated by the modules in $\calS$. For this it is sufficient to show that the injective $k\calP$-modules lie in the subcategory generated by $\calS$. Evidently $P_\alpha$ is one of these. If $I$ is any other injective $k\calP$-module there is a surjection $P_\alpha/\Soc P_\alpha\to I$ whose kernel is a direct sum of modules, each of which is a $k[a_i,z_i]$-module for some $i$. We know that the modules in the $\calS_i$ generate such modules by hypothesis. This shows that the modules in $\calS$ generate $D^b(k\calP)$.

Finally we observe that the direct sum of modules in $\calS$ is a perfect complex since $k\calP$ has finite global dimension.

(4) The only non-zero homomorphisms between $P_\alpha$, $P_\alpha/\Soc P_\alpha$ and members of $\calS$ are the endomorphisms of these modules, a 1-dimensional space of homomorphisms $P_\alpha\to P_\alpha/\Soc P_\alpha$, and a 1-dimensional space of homomorphisms $k_{[a_i,z_i]}\to P_\alpha/\Soc P_\alpha$ for each $i$. There are no other homomorphisms than these because $P_\alpha$ has $\omega$ as its socle and no other module in $\calS$ has $\omega$ as a composition factor, $P_\alpha$ and $P_\alpha/\Soc P_\alpha$ have $\alpha$ as their top and no other modules have $\alpha$ as a composition factor, and $P_\alpha/\Soc P_\alpha$ has $z_1,\ldots,z_n$ forming its socle, while none of the modules in $\calS$ except   $P_\alpha$, $P_\alpha/\Soc P_\alpha$ and $k_{[a_i,z_i]}$ have such a composition factor. The same observations show that $\dim\End_{D^b(k\calP)}(P_\alpha)=\dim\End_{D^b(k\calP)}(P_\alpha/\Soc P_\alpha)=1$. This analysis of the homomorphisms shows, by induction, that $\End_{D^b(k\calP)}(U)$ is the path algebra of the quiver obtained by joining  a star with center corresponding to $P_\alpha$ and arms $\bullet\to \bullet\gets \bullet$ to the quivers coming by induction from $\calS_i$.
\end{proof}

Proposition~\ref{inductive-clamping-proposition} was an inductive step that we use to describe part of the Auslander-Reiten quiver of posets obtained by iterative clamping. We already see it illustrated in Example~\ref{A4-example}, where the quiver of the poset with tree class $A_4$ is obtained from the clamped interval with tree class $A_2$ by joining on two more edges.

It is convenient to refer to the \textit{wing} of a quiver component determined by a vertex $x$ in that component. By this we mean the region of the quiver component whose vertices are those that can appear in some slice together with $x$. It is bounded on one side by the vertices $y$ for which there is a chain of irreducible morphisms $y=y_0\to y_1\to y_2\to \cdots\to y_n=x$ with $y_i\ne \tau y_{i+2}$ always, and on the other side by vertices $y$ for which there is a chain of irreducible morphisms $x=y_0\to y_1\to y_2\to \cdots\to y_n=y$ satisfying the same condition.

\begin{corollary}\label{tree-construction-corollary}
Let $\calP\in\calI\calC$ be a poset obtained by iterative clamping. Then
\begin{enumerate}
\item $\calP$ satisfies Hypothesis~\ref{wing-hypothesis}.
\item $k\calP$ is piecewise hereditary.
\item The component of the Auslander-Reiten quiver of $D^b(k\calP)$ that contains the largest projective $P_\alpha$ has the form $\ZZ T$ where the tree $T$ is obtained inductively by taking a star with an edge for each interval clamped at the top level, plus an extra edge whose free vertex is labelled with $P_\alpha$, and joining to each of the remaining free vertices the trees for the clamped intervals, joining at the vertices corresponding to the largest projectives. 
\item The algebra $k\calP$ is derived equivalent to the path algebra of any quiver whose underlying graph is the tree $T$. 
\item The wing
of the quiver component determined by $P_\alpha$ consists entirely of modules.

\end{enumerate}
\end{corollary}

\begin{proof}
The proof is by induction on $n$, where $\calP\in\calI\calC_n$. When $n=0$ the poset consists of just a single point, and the result is immediate because $k\calP=k$ is a field. Parts (1), (2), (3) and (4) now follow using Proposition~\ref{inductive-clamping-proposition} as the induction step. The argument is that we obtain a slice of the quiver component containing $P_\alpha$ with underlying tree $T$, the direct sum of whose terms is a tilting complex with endomorphism ring a hereditary algebra, necessarily of tree class $T$. Thus the Auslander-Reiten quivers of the bounded derived categories of $k\calP$ and this hereditary algebra are isomorphic. It follows that the quiver component for $k\calP$ containing $P_\alpha$ is $\ZZ T$ (rather than some quotient of this) because of the description of the quiver component containing projectives of a hereditary algebra given in~\cite{Hap}.

The argument that proves (5) also proceeds by induction. We merely sketch how it goes, to avoid some of the technicalities. The wings of all of the largest projectives of the clamped intervals, taken in the quivers of those intervals, consist entirely of modules, by induction; and by an application of Corollary~\ref{AR-quiver-copying-corollary} (after the conditions on the support of these modules are checked) are copied into the quiver component of $P_\alpha$. Note that none of these modules have $\alpha$ or $\omega$ in their support, and so cannot be projective or injective. The wing of $P_\alpha$ consists of all the modules just considered, together with complexes $M$ for which there is a directed path from $M$ to $P_\alpha$ or from $P_\alpha$ to $M$. The arguments in these two cases are similar, and we show by induction on the distance between
$M$ and $P_\alpha$ that $M$ is a module. The argument is that $M$ appears as the first or third term in an Auslander-Reiten triangle where the other two remaining terms are modules. We use the long exact sequence in homology to show that $M$ has homology only in degree 0, using also a result we have not yet come to: Lemma~\ref{split-long-exact-sequence-lemma}.  To apply that lemma we must verify that the third term of the Auslander-Reiten triangle is not projective. When $M$ is the first term, the third term is one of the copied modules that does not have $\omega$ in its support, so the third term is not projective. If $M$ were the third term and projective, the first term is the $-1$ shift of an injective and  also one of the copied modules, which is not possible. \end{proof} 

More generally, the inductive argument just indicated shows that, in the wing of $P_\alpha$, projectives can only appear on the `left boundary' of the wing, and injectives only on the `right boundary' of the wing.  Thus the `interior' of the wing of $P_\alpha$ contains no projectives or injectives.





We will present various applications of Corollary~\ref{tree-construction-corollary}. The first is a determination of when $k\calP$ is fractionally Calabi-Yau if $\calP$ is obtained by iterative clamping.

\begin{corollary}
\label{Calabi-Yau-IC-corollary}
Let $\calP\in\calI\calC$ be a poset obtained by iterative clamping. Then $k\calP$ is fractionally Calabi-Yau if and only if the tree constructed in Corollary~\ref{tree-construction-corollary} is a Dynkin diagram.
\end{corollary}

\begin{proof}
This is immediate from Scherotzke's result \cite[4.13]{Sch} since for $k\calP$ to be fractionally Calabi-Yau the tree class of each Auslander-Reiten component must be a Dynkin diagram or $A_\infty$, and the tree constructed in Corollary~\ref{tree-construction-corollary} is always finite.
\end{proof}

\begin{example}
\label{star-example}
Let $\calP(q_1,q_2,\ldots,q_t)$ be the poset that consists of $t$ chains of lengths $q_1,q_2,\ldots,q_t$, all of which are clamped intervals lying between an initial vertex $\alpha$ and a terminal vertex $\omega$. For example, the following poset has a chain of length 1 and a chain of length 2 that are both clamped intervals, lying between an initial vertex $\alpha$ and a terminal vertex $\omega$.
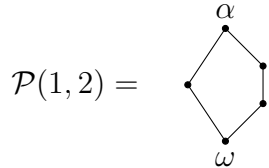
\begin{center}
\begin{tikzpicture}[xscale=.5,yscale=.5]
\draw (1,0)--(0,1.5)--(1,3)--(2,2)--(2,1)--(1,0);
\draw[fill] (1,0) circle [radius=0.08];
\draw[fill] (0,1.5) circle [radius=0.08];
\draw[fill] (1,3) circle [radius=0.08];
\draw[fill] (2,2) circle [radius=0.08];
\draw[fill] (2,1) circle [radius=0.08];

\node [ below] at  (1,0)  {$\omega$};
\node [ above] at  (1,3)   {$\alpha$};
\node at (-3,1.5) {$\calP(1,2)=$};
\end{tikzpicture}
\end{center}
Then $k\calP(q_1,q_2,\ldots,q_t))$ is derived equivalent to the path algebra of any quiver whose underlying graph is a star, consisting of $t+1$ arms of lengths $2,q_1+1,q_2+1,\ldots,q_t+1$ joined at a single vertex. It is clear that we can obtain any star in this way provided one of the arms has length 2. (We will see in a moment how to obtain an arbitrary star by a modification of the above construction.)

For example, the Auslander-Reiten quiver component of $D^b(\calP(1,2))$ that contains the largest projective module $P_\alpha$ has tree class that is a star with arms of length 2, 2 and 3. This is the Dynkin diagram $D_5$.
Furthermore, $k\calP(q_1,q_2,\ldots,q_t)$ is fractionally Calabi-Yau if and only if $(q_1,\ldots,q_t)= (q_1), (1,q_2)$ or $(2,2), (2,3), (2,4)$, by Corollary~\ref{Calabi-Yau-IC-corollary}. In these cases the Auslander-Reiten quiver of $D^b(\calP(q_1,q_2,\ldots,q_t))$ consists of a single component that has tree class $A_{q_1+2}, D_{q_2+3}, E_6, E_7, E_8$.
\end{example}

So as to obtain a greater variety of tree classes we now describe a modification of the construction of the posets in $\calI\calC$. The modification is to adjoin only a new minimal element $\alpha$ to a poset $[a,z]$ that already has unique minimal element $a$ and unique maximal element $z$. (We could equally adjoin a new maximal element $\omega$, leaving $a$ as the minimal element.) The interval $[a,z]$ is clamped in the new poset and we can apply Proposition~\ref{glue-proposition} again.

\begin{proposition}\label{modified-clamping-proposition}
Let $\calP=\{\alpha\}\cup [a,z]$ be a poset that is the union of an interval $[a,z]$ and another element $\alpha$ with $\alpha<a$. Thus $\alpha$ is the unique minimal element and $z$ is the unique maximal element.  Suppose that the Auslander-Reiten quiver component of $D^b(k[a,z])$ that contains the projective $P_{a}$  has a slice $\calS$ with tree class $T$ satisfying Hypothesis~\ref{wing-hypothesis}. Assume, furthermore, that $P_a/P_z$ lies in $\calS$.
\begin{enumerate}
\item 
Then $\calP$ also satisfies Hypothesis~\ref{wing-hypothesis} with the slice
$$
\calS^+=(\calS-\{P_a\})\cup \{P_\alpha, P_\alpha/P_z\}.
$$
\item The tree class $T^+$ of the Auslander-Reiten quiver component of $D^b(k\calP)$ that contains the largest projective $P_\alpha$ is obtained from $T$ by gluing an extra edge to the point corresponding to $P_a$.
\item If the direct sum of modules in $\calS$ is a tilting complex $U$ for the interval $[a,z]$, then the direct sum of modules in the slice $\calS^+$ is a tilting complex $V$ for $k\calP$. 
\item If $\End_{D^b(k[a,z])}(U)$ is hereditary then so is $\End_{D^b(k\calP)}(V)$, and so $k\calP$ is piecewise hereditary.
\end{enumerate}
\end{proposition}

\begin{proof}
The proof is very similar to that of Proposition~\ref{inductive-clamping-proposition}. We start differently in that, by Proposition~\ref{glue-proposition}, we now have an Auslander-Reiten triangle in $D^b(k\calP)$ with terms
$$
P_a\to P_\alpha\oplus (P_a/P_z)\to P_{\alpha}/P_z\to P_a[1],
$$
noting that $P_a=k_{[a,z]}$, $P_z=\Soc P_a=\Soc P_\alpha$, $P_a/P_z=k_{[a,z)}$ and $P_\alpha/P_z=k_{[\alpha,z)}$. We wish to obtain a slice consisting of modules and containing $P_\alpha$, so that the only module in the slice with $z$ in its support is $P_\alpha$. Since $\calS$ contains $P_a$, which has $z$ in its support, we get rid of it and replace it by $\tau^{-1}P_a=P_{\alpha}/P_z$. It then suffices to adjoin one new module $P_\alpha$ to obtain a slice $\calS^+$. From the construction we see it satisfies Hypothesis~\ref{wing-hypothesis} and condition (2).

We must now establish (3) and (4), and we leave the details of this to the reader. The techniques used are exactly the same as in the proof of Proposition~\ref{inductive-clamping-proposition} with some different modules appearing that come from the slightly different situation.
\end{proof}

As an example of the situation described in Proposition~\ref{modified-clamping-proposition}, the particular case of Example~\ref{A4-example} shows explicitly how the tree class of the quiver when $\calP$ is a chain of length 4 is obtained from the tree class of the quiver of a chain of length 3 by adding a single new vertex (as well as in the manner described in Proposition~\ref{inductive-clamping-proposition}).

The construction of the class of posets $\calI\calC$ was done so as to facilitate the inductive argument of Proposition~\ref{inductive-clamping-proposition}. Evidently we can enlarge $\calI\calC$ by allowing the further operation of Proposition~\ref{modified-clamping-proposition}, giving a class of posets  $\calI\calC^+$ that are also piecewise hereditary. We now show how to construct lattices in $\calI\calC^+$ with certain specified tree classes. We will refer to the \textit{largest indecomposable projective} of a lattice $\calP$: this is the module $k_\calP$.

\begin{theorem}
\label{lattice-with-specified-tree-class}
Let $T$ be a finite tree with the property that whenever $x,y$ are vertices with valence $\ge 3$ then $d(x,y)\ge 2$. Then there is a lattice $\calP\in\calI\calC^+$
with the properties:
\begin{enumerate}
\item the component of the Auslander-Reiten quiver of $D^b(k\calP)$ containing the largest indecomposable projective $k\calP$-module has tree class $T$, 
\item the largest indecomposable projective $k\calP$-module may lie in the $\tau$-orbit corresponding to any specified extreme vertex of $T$, and
\item the wing of the quiver component of $D^b(k\calP)$ determined by the largest indecomposable projective $k\calP$-module consists entirely of modules.
\end{enumerate}
\end{theorem}

\begin{proof}
We proceed by induction on the number $v$ of vertices of $T$ of valence $\ge 3$. When this number is 0, $T$ is the Dynkin diagram $A_n$ for some $n$. Letting $\calP$ be a chain with $n$ elements produces the desired result.

When $v=1$, $T$ is a star and we have seen in Example~\ref{star-example} lattices with the required properties.

Now let $v>1$ and suppose the result is proved for smaller values of $v$. Let $p$ be the extreme vertex of $T$ that is to correspond to the largest projective $k_\calP$ and let $x$ be the vertex of $T$ of valence $\ge 3$ closest to $p$. Let the edges of $T$ incident with $x$ be $e_0,e_1,\ldots,e_r,$ where $e_0$ lies on the path between $p$ and $x$, and label the vertices of $e_i$ by $x$ and $y_i$. Removing from $T$ the star consisting of the path from $x$ to $p$ and the edges $e_1,\ldots,e_r,$ yields a disjoint union of $r$ trees $T_1,\ldots, T_r$, and $y_i$ is an extreme vertex of $T_i$ for each $i$. By induction we can find a lattice $\calP_i$ satisfying the condition of the theorem with tree $T_i$ and largest projective $k_{\calP_i}$ corresponding to $y_i$. Clamping $\calP_1,\ldots ,\calP_r$ between elements $\alpha'$ and $\omega$ produces a lattice $\calP'$ so that the component of the Auslander-Reiten quiver of $D^b(k\calP')$ containing the largest projective has tree class the subtree $T'$ of $T$ that consists of $T$ minus the edges from $y_0$ to $p$, and with the largest projective associated to $y_0$. Finally, by adjoining a chain of elements in bijection with the vertices of $T-T'$ to the smallest element $\alpha'$ of $\calP'$ produces a poset $\calP$ with tree $T$ as required, by Propositions~\ref{inductive-clamping-proposition} and \ref{modified-clamping-proposition}.
This completes the inductive step. The final statement about the wing of the largest projective is proved in the same way as part (4) of Corollary~\ref{tree-construction-corollary}
\end{proof}

\section{The connection with the Auslander-Reiten quiver of $k\calP$-mod}

We explore the connection between Auslander-Reiten triangles in the bounded derived category and in the module category. It is tempting to suppose that if we take an Auslander-Reiten triangle in $D^b(\Lambda)$, where $\Lambda$ is a finite dimensional algebra, and evaluate its homology in some dimension, we will get either a split short exact sequence, or an Auslander-Reiten sequence direct sum a split sequence. This is not true in general: the Auslander-Reiten triangle ending in a projective module has zero homology that does not produce a short exact sequence. 

We start by providing a sufficient condition for the homology of an Auslander-Reiten triangle in some dimension to be an Auslander-Reiten sequence direct sum a split sequence, or split. This approach will provide us with a method to compute the Auslander-Reiten quiver of $k\calP$-mod, starting from the Auslander-Reiten quiver of $D^b(k\calP)$. Our first results hold in the generality of representations of a finite dimensional algebra and have wider interest than just representations of posets.

We denote the Auslander-Reiten translate by $\tau$, and for any object $M$ we refer to the \textit{$\tau$-orbit of $M$} as the set of objects of the form $\tau^i(M)$ for some $i\in\ZZ$. This is an abuse of terminology because $\tau$ may only be partially defined on a $\tau$-orbit, so that the operation of applying $\tau$ might not be invertible.

The first lemma follows from \cite[2.1]{Web2}; apart from this we have been unable to find references for it in the literature.

\begin{lemma}\label{split-long-exact-sequence-lemma}
Let $\Lambda$ be a finite dimensional algebra and let
$$
L\to M\to N\to L[1]
$$
be an Auslander-Reiten triangle in $D^b(\Lambda)$. Assume that $N$ is not the shift of a projective module (a complex with only one nonzero term). Then the associated long exact homology sequence is the splice of short exact sequences
$$0\to H_i(L)\to H_i(M)\to H_i(N)\to 0$$
with zero connecting homomorphisms.
\end{lemma}

\begin{proof}
The homology $H_i$ is a representable functor:
$$
H_i(N)=\Hom_{D^b(\Lambda)}(\Lambda[i],N).
$$
The connecting homomorphisms in the long exact sequence of homology arising from the triangle are all zero, except when $N$ is the shift a projective module, because of the lifting property of the Auslander-Reiten triangle, since $N$ is never a summand of $\Lambda[i]$.
\end{proof}

A similar analysis shows that, when $N$ is the shift of a projective module, the long exact sequence in homology of the Auslander-Reiten triangle is not a splice of short exact sequences. We now describe circumstances in which we obtain short homology sequences that are Auslander-Reiten sequences, up to split summands.

\begin{proposition}
\label{AR-triangle-sequence-proposition}
Let $\Lambda$ be a finite dimensional algebra and let
$$L\to M\to N\to L[1]$$
be an Auslander-Reiten triangle in $D^b(\Lambda)$. Suppose that $N$ is a truncated projective resolution of some non-projective $\Lambda$-module. Thus $N$ has at most two non-zero homology groups, one in degree 0 and possibly one more in a higher degree at the other end of the complex, and is not the shift of a projective module. Then
$$
0\to H_0(L)\to H_0(M)\to H_0(N)\to 0
$$
is a short exact sequence of $\Lambda$-modules that is the direct sum of a split exact sequence and (possibly) an Auslander-Reiten sequence.
\end{proposition}

\begin{proof} 
Exactness is proved in Lemma~\ref{split-long-exact-sequence-lemma}. 
Note that $N$ is indecomposable, by the Auslander-Reiten triangle hypothesis. It follows that $H_0(N)$ is indecomposable, since $N$ is a truncated projective resolution of $H_0(N)$, and this would decompose if $H_0(N)$ did. We show that if $f:U\to H_0(N)$ is a module homomorphism that is not a split epimorphism then $f$ lifts to a homomorphism $U\to H_0(M)$. 
Let $P_U$ be the start of a projective resolution of $U$ with the same number of non-zero terms as the complex $N$. Then $f$ lifts to a homomorphism of complexes $\phi:P_U\to N$, so that $H_0(\phi)=f$, and $\phi$ is not a split epimorphism because $f$ is not a split epimorphism. Thus $\phi$ lifts to a homomorphism of complexes $P_U\to M$ by the Auslander-Reiten property. 
Taking zero homology we get that $f$ lifts through $H_0(M)$. From this it follows that $0\to H_0(L)\to H_0(M)\to H_0(N)\to 0$ is a direct sum of a split short exact sequence and (possibly) an Auslander-Reiten sequence of modules.
The argument is standard: either $H_0(M)\to H_0(N)$ is split epi, or else it is right almost split, in which case it breaks up as a direct sum of a minimal right almost split map and a zero map (see~\cite{ARS}). In the first case the sequence is split and in the second case it is the direct sum of a split sequence and Auslander-Reiten sequence.
\end{proof}

\begin{corollary}
\label{triangle-to-sequence-copying}
Let $\Lambda$ be a finite dimensional algebra.
\begin{enumerate}
\item Suppose that $L\to M\to N\to L[1]$ is an Auslander-Reiten triangle in $D^b(\Lambda)$ where $L$, $M$ and $N$ are all isomorphic to modules (i.e. their only non-zero homology is in degree 0). Then the sequence of $\Lambda$-modules $0\to H_0(L)\to H_0(M)\to H_0(N)\to 0$ is an Auslander-Reiten sequence.
\item If part of the Auslander-Reiten quiver of $D^b(\Lambda)$ consists of modules, forming a union of meshes, then this part is copied into the Auslander-Reiten quiver of $\Lambda$-mod.
\end{enumerate}
\end{corollary}

\begin{proof} (1) Note that $N$ cannot be a projective module (or the shift of a projective module) because if it were, then $L \cong \nu N[-1]$ would fail to be a module, contrary to the assumption on $L$.
Thus, by Proposition~\ref{AR-triangle-sequence-proposition}, $0\to H_0(L)\to H_0(M)\to H_0(N)\to 0$ is a short exact sequence of modules that is a direct sum of a split sequence and possibly an Auslander-Reiten sequence. Since $L$ and $N$ are indecomposable, so are the modules $H_0(L)$ and $H_0(N)$, because these are the only non-zero homology modules. For the same reason, the sequence  $0\to H_0(L)\to H_0(M)\to H_0(N)\to 0$ is not split. We deduce that the sequence is an Auslander-Reiten sequence.

(2) is an immediate consequence.
\end{proof}

\begin{corollary}
\label{wing-copy-corollary}
Let $\calP\in\calI\calC$ be a poset obtained by iterative clamping.
The wing of the quiver component of $D^b(k\calP)$ containing the largest projective $P_\alpha$ consists entirely of modules, and is a union of meshes.
A copy of this wing appears in the Auslander-Reiten quiver of $k\calP\hbox{-mod}$.
\end{corollary}

\begin{proof}
We have seen in Corollary~\ref{tree-construction-corollary} that the wing of the quiver of $D^b(k\calP)$ determined by the largest projective consists of modules, and from the definition of the wing it is a union of meshes, except in the case when $\calP$ is a single point. In that case $k\calP=k$ and the result is true. It follows from Corollary~\ref{triangle-to-sequence-copying} the wing is copied into the Auslander-Reiten quiver of $k\calP\hbox{-mod}$.
\end{proof}

We now show how to use the information from Corollary~\ref{wing-copy-corollary} to help compute the Auslander-Reiten quiver of $k\calP\hbox{-mod}$ for a poset $\calP$ obtained by iterative clamping. Note that something more general can also be done for an extended class of posets using Proposition~\ref{modified-clamping-proposition}, but we leave this to the reader. We have seen how to obtain a part of a quiver component of $k\calP\hbox{-mod}$, but this component will not be stable, because it contains the projective module $P_\alpha$. To describe the whole component we consider the stable part as well as $P_\alpha$ and the $\tau$-orbits of the finitely many other projectives and injectives that may happen to be there.

By slight extension of a notation used previously, if $\theta$ is an element of $\calP$ we denote by $k_\theta$ the simple $k\calP$-module supported at $\theta$, with dimension 1 at $\theta$ and 0 at other elements of $\calP$. We will make use of additive functions on the Auslander-Reiten quiver of $k\calP$-mod. Writing the largest element of $\calP$ as $\omega$, as before, let $f_\omega$ be the function defined on $k\calP$-modules given by $f_\omega(M):= \dim_k\Hom(P_\omega, M)$, the composition factor multiplicity of $k_\omega$ in $M$. Since composition factor multiplicities are additive on short exact sequences, this function is additive on every mesh in the Auslander-Reiten quiver of $k\calP$-mod. Given a translation quiver with Auslander-Reiten translate $\tau$, we term a \textit{$\tau$-orbit interval} a set of vertices of the form $\tau^i M$ with $r\le i\le s$ for some integers $r$ and $s$.

\begin{proposition}
\label{AR-quiver-subset-of-ZT}
Let $\calP\in\calI \calC$ be a poset obtained by iterated clamping. Thus $\calP$ has a smallest element $\alpha$, a largest element $\omega$, and satisfies the hypotheses of Proposition~\ref{inductive-clamping-proposition}, with a slice $\calS$ and underlying tree $T$ as specified there. Then the component of the Auslander-Reiten quiver of $k\calP$-mod that contains $P_\alpha$ is a subset of $\ZZ T$. Identifying the modules in $\calS$ with a copy of $T$ in $\ZZ T$, the subset  of $\ZZ T$ that is the Auslander-Reiten quiver of $k\calP$-mod is a union of $\tau$-orbit intervals, each interval containing an element of $\calS$ and either infinite or terminating in a  projective module to the left, or an injective module to the right. 
\end{proposition}

\begin{proof}
The function $f_\omega$ is 0 on $\calS$ except on $P_\alpha$, where it takes the value 1. This is because of the hypothesis that the only module in $\calS$ with support on $\omega$ is $P_\alpha$. There is no mesh starting at $P_\alpha$, because it is injective. If a new projective module $P_x$ appears in this quiver component to the right of $\calS$, then it must arise via the only irreducible morphism terminating in $P_x$, namely $\Rad P_x\to P_x$. Notice that $f_\omega(\Rad P_x)=f_\omega(P_x)=1$, so that $f_\omega$ must already have taken the value 1 to the right of $S$ before the projective $P_x$ appears. If $\calS$ has at least three elements (the case of one or two is left to the reader), there is a mesh starting at an element of $\calS$ with its middle terms in $\calS$, perhaps apart from a projective. There can, in fact, be no projective, because $f_\omega$ is zero on the left of the mesh. Thus $f_\omega$ is zero at the right term of this mesh. Repeating in this way we deduce that $f_\omega$ is zero everywhere to the right of $\calS$ and that no new projective modules $P_x$ appear there. From this we see that the quiver to the right of $\calS$ consists entirely of negative powers of $\tau$ applied the non-injective terms in $\calS$, with a finite $\tau$-orbit interval to the right if it is terminated by an injective module. Because the hypothesis that $\calP\in \calI\calC$ implies that the opposite poset $\calP^{\rm op}\in \calI \calC$, we see dually that no injective can appear to the left of $\calS$ and that every module in this quiver component has the form $\tau^n M$ for some $n\in\ZZ$ and $M\in\calS$ in an interval that may be terminated on the left by a projective, and on the right by an injective.
\end{proof}

The calculation of Auslander-Reiten quivers of $k\calP$-mod in our examples is facilitated by the following proposition. 

\begin{proposition}
\label{tau-calculations}
\begin{enumerate}
\item Suppose the poset $\calP$ consists of two clamped intervals $[\beta,\delta]$ and $[\gamma,\epsilon]$ lying between an initial element $\alpha$ and a terminal element $\omega$. Then the Auslander-Reiten translate on $k\calP$-mod has values $\tau k_\epsilon=P_\beta$ and $\tau k_\delta=P_\gamma$.
\item Whenever a chain $a<b<c<\cdots<z$ is clamped we have $\tau k_a=k_b$, $\tau k_b=k_c$, and so on.
\end{enumerate}
\end{proposition}

\begin{proof}
(1) We calculate the effect of $\tau$ on $k_\epsilon$ by taking the first two terms of a projective resolution $P_\omega\to P_\epsilon\to k_\epsilon\to 0$ (in fact this is the complete resolution) and applying the Nakayama functor to get $I_\omega\to I_\epsilon$. The kernel of this map is $P_\beta$, so $\tau k_\epsilon=P_\beta$. The calculation of $\tau k_\delta$ is similar.

(2) We calculate the resolution $P_b\to P_a\to k_a\to 0$, apply the Nakayama functor to get $I_b\to I_a$ and observe that the kernel is $k_b$.
\end{proof}

\begin{example}
\label{AR-quiver-example}
We are ready to give a fairly elaborate example of the application of these methods. We calculate the Auslander-Reiten quivers of both $D^b(k\calP)$ and $k\calP$-mod when $\calP$ is the poset with Hasse diagram:
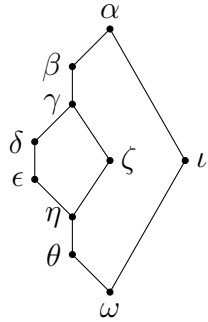
\begin{center}
\begin{tikzpicture}[xscale=.5,yscale=.5]
\draw (2,0)--(1,1)--(1,2)--(0,3)--(0,4)--(1,5)--(1,6)--(2,7)--(4,3.5)--(2,0);
\draw (1,2)--(2,3.5)--(1,5);
\draw[fill] (2,0) circle [radius=0.08];
\draw[fill] (1,1) circle [radius=0.08];
\draw[fill] (1,2) circle [radius=0.08];
\draw[fill] (0,3) circle [radius=0.08];
\draw[fill] (0,4) circle [radius=0.08];
\draw[fill] (1,5) circle [radius=0.08];
\draw[fill] (1,6) circle [radius=0.08];
\draw[fill] (2,7) circle [radius=0.08];
\draw[fill] (4,3.5) circle [radius=0.08];
\draw[fill] (2,3.5) circle [radius=0.08];

\node [ below] at  (2,0)  {$\omega$};
\node [ left] at  (1,1)   {$\theta$};
\node [ left] at  (1,2)   {$\eta$};
\node [ left] at  (0,3)   {$\epsilon$};
\node [ left] at  (0,4)   {$\delta$};
\node [ left] at  (1,5)   {$\gamma$};
\node [ left] at  (1,6)   {$\beta$};
\node [ above] at  (2,7)   {$\alpha$};
\node [ right] at  (4,3.5)   {$\iota$};
\node [ right] at  (2,3.5)   {$\zeta$};
\end{tikzpicture}
\end{center}
It is remarkable that we will compute the quiver of $k\calP$-mod without computing the structure of every module in detail. 
We see that $\calP$ is obtained by iterative clamping, starting from the intervals $[\delta,\epsilon]$, $[\zeta,\zeta]$ and $[\iota,\iota]$ and so, according to Proposition~\ref{inductive-clamping-proposition}, $k\calP$ piecewise hereditary, being derived equivalent to the path algebra of a quiver with underlying tree $T$:
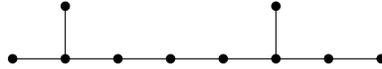
\begin{center}
\begin{tikzpicture}[scale=.7]
\draw (0,0)--(1,0)--(2,0)--(3,0)--(4,0)--(5,0)--(6,0)--(7,0);
\draw (1,0)--(1,1); \draw (5,0)--(5,1);
\foreach \y in {0,1,2,3,4,5,6,7}
\draw[fill] (\y,0) circle [radius=0.08];
\draw[fill] (1,1) circle [radius=0.08];
\draw[fill] (5,1) circle [radius=0.08];
\end{tikzpicture}
\end{center}
which is of wild representation type. On the other hand we claim that $k\calP$ is of finite representation type, and that its Auslander-Reiten quiver has shape indicated by the following diagram:

\begin{center}
\begin{tikzpicture}[scale=.5]
\draw[dashed] (0,-3)--(0,8);
\foreach \y in {2,4,6}
\draw[shift={(\y,0)}]
(-17,6) -- (-12,1);
\draw (-17,4)--(-14,1); \draw (-18,3)--(-16,1); 
\draw  (-15,6)--(-19,2)--(-18,1);
\draw (-16,5)--(-12,1);
\foreach \y in {0,2,4}
\draw[shift={(\y,0)}]
(-9,6) -- (-2,-1);
\foreach \y in {0,2,4,6,8,10,12,14}
\draw[shift={(\y,0)}]
(-18,1) -- (-13,6);
\draw (2,-1)--(4,1)--(-1,6);
\draw (0,-1)--(4,3)--(1,6);
\draw (-2,-1)--(4,5)--(3,6);
\draw (-3,0)--(3,6);
\draw (-3,6)--(3,0);
\draw (-17,4)--(4,4);
\draw (-1,0)--(1,0);
\foreach \y in {-16,-14,-12,-10,-8,-6,-4,-2,0,2,4}
\draw[fill] (\y,4) circle [radius=0.05];
\draw[fill] (0,0) circle [radius=0.05];
\foreach \y in {-11,-5,-3,-1,1,3}
\node [below] at (\y,6.1) {$\scriptstyle 0$};
\foreach \y in {-15,-13,-9,-7}
\node [below] at (\y,6.1) {$\scriptstyle 1$};
\foreach \y in {-4, -2, 0, 2, 4}
\node [below] at (\y,5.1) {$\scriptstyle 0$};
\foreach \y in {-16,-12,-10,-6}
\node [below] at (\y,5.1) {$\scriptstyle 1$};
\foreach \y in {-14,-8}
\node [below] at (\y,5.1) {$\scriptstyle 2$};
\foreach \y in {-4, -3,-2,-1,0,1,2,3,4}
\node [below] at (\y,4.1) {$\scriptstyle 0$};
\foreach \y in {-17,-16,-14,-12,-10,-8,-6,-5}
\node [below] at (\y,4.1) {$\scriptstyle 1$};
\foreach \y in {-15,-13,-11,-9,-7}
\node [below] at (\y,4.1) {$\scriptstyle 2$};
\foreach \y in {-2,0,2,4}
\node [below] at (\y,3.1) {$\scriptstyle 0$};
\foreach \y in {-18,-16,-14,-8,-6,-4}
\node [below] at (\y,3.1) {$\scriptstyle 1$};
\foreach \y in {-12,-10}
\node [below] at (\y,3.1) {$\scriptstyle 2$};
\foreach \y in {-15,-7,-1,1,3}
\node [below] at (\y,2.1) {$\scriptstyle 0$};
\foreach \y in {-19,-17,-13,-9,-5,-3}
\node [below] at (\y,2.1) {$\scriptstyle 1$};
\foreach \y in {-11}
\node [below] at (\y,2.1) {$\scriptstyle 2$};
\foreach \y in {-16,-14,-8,-6,0,2,4}
\node [below] at (\y,1.1) {$\scriptstyle 0$};
\foreach \y in {-18,-12,-10,-4,-2}
\node [below] at (\y,1.1) {$\scriptstyle 1$};
\foreach \y in {1,3}
\node [below] at (\y,.1) {$\scriptstyle 0$};
\foreach \y in {-3,-1,0}
\node [below] at (\y,.1) {$\scriptstyle 1$};
\foreach \y in {0,2}
\node [below] at (\y,-0.9) {$\scriptstyle 0$};
\foreach \y in {-2}
\node [below] at (\y,-0.9) {$\scriptstyle 1$};

\node [ below] at (0,-0.9)  {$\scriptstyle 0$};
\end{tikzpicture}
\end{center}
In this diagram the vertical dashed line indicates that the picture continues to the right as a reflection of the part to the left of the vertical line. The numbers  attached to vertices are the values of the  function $f_\omega$, which is not invariant under reflection in the dashed line.

To show that this is the Auslander-Reiten quiver of $k\calP$-mod, we observe that the quiver component containing $P_\alpha$ is a subset of $\ZZ T$ by Proposition~\ref{AR-quiver-subset-of-ZT}. By Proposition~\ref{tau-calculations} and Corollary~\ref{triangle-to-sequence-copying} the following is part of this quiver component, because it is part of the quiver of $D^b(k\calP)$ and all the terms are modules:
\begin{center}
{
\def\diagram#1{\null\,\vcenter{\tempbaselines
\mathsurround=0pt
    \ialign{\hfil$##$\hfil&&\hskip15pt\hfil$##$\hfil\crcr
      \mathstrut\crcr\noalign{\kern-\baselineskip}
  #1\crcr\mathstrut\crcr\noalign{\kern-\baselineskip}}}\,}
$\diagram{&&&&&&\clap{k_\epsilon}&&&&\clap{k_\delta}\cr
&&&&&&&\searrow&&\nearrow\cr
&&&&&&&&\clap{k_{[\delta,\epsilon]}}&\cr
&&&&&&&\nearrow&&\searrow\cr
&&&&&&\clap{\Rad k_{[\gamma,\eta]}}&\to&\clap{k_\zeta}&\to&\rlap{$k_{[\gamma,\eta]}/\Soc k_{[\gamma,\eta]}$}\cr
&&&&&&&\searrow&&\nearrow\cr
&&&&&&&&\clap{k_{[\gamma,\eta]}}&\cr
&&&&&&&\nearrow&&\searrow\cr
&&&&&&\clap{k_{[\gamma,\theta]}}&&&&\clap{k_{[\beta,\eta]}}\cr
&&&&&&&\searrow&&\nearrow\cr
\clap{\Rad P_\gamma}&&&&\clap{?}&&&&\clap{k_{[\beta,\theta]}}&&&\cr
&\searrow&&\nearrow&&\searrow&&\nearrow&&\searrow\cr
&&\clap{P_\gamma}&&&&\clap{\Rad P_\alpha}&\to&\clap{P_\alpha}&\to&\rlap{$P_\alpha/\Soc P_\alpha$}\cr
&&&\searrow&&\nearrow&&\searrow&&\nearrow\cr
&&&&\clap{P_\beta}&&&&\clap{k_\iota}&\cr
}
$
}
\end{center}
We write ? to denote a module whose precise structure has not been determined.
Notice that the remaining indecomposable projectives not listed in the above diagram have the property that their radical is also projective, and so the following is part of the Auslander-Reiten quiver of $k\calP$-mod:
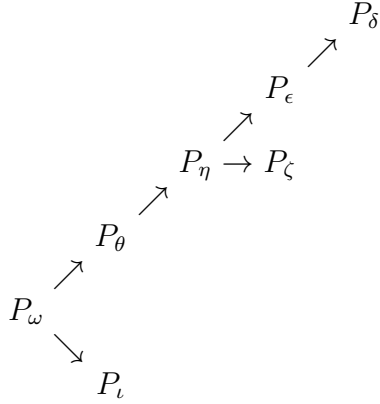
\begin{center}
{
\def\diagram#1{\null\,\vcenter{\tempbaselines
\mathsurround=0pt
    \ialign{\hfil$##$\hfil&&\hskip10pt\hfil$##$\hfil\crcr
      \mathstrut\crcr\noalign{\kern-\baselineskip}
  #1\crcr\mathstrut\crcr\noalign{\kern-\baselineskip}}}\,}
$\diagram{&&&&&&&&&&\clap{P_\delta}\cr
&&&&&&&&&\nearrow\cr
&&&&&&&&\clap{P_\epsilon}&\cr
&&&&&&&\nearrow&&\cr
&&&&&&\clap{P_\eta}&\to&\clap{P_\zeta}&&\cr
&&&&&\nearrow&&&&\cr
&&&&\clap{P_\theta}&&&&&\cr
&&&\nearrow&&&&&&\cr
&&\clap{P_\omega}&&&&&&&&\cr
&&&\searrow&&&&&&\cr
&&&&\clap{P_\iota}&&&&&&&\cr
}
$
}
\end{center}
From this we deduce that these projectives all lie in the same quiver component, so either they are all in the same component as $P_\alpha$ or none of them is. We conclude that the $\tau$-orbits in the component of $P_\alpha$ that do not terminate in $P_\alpha$, $P_\beta$ or $P_\gamma$ are either all infinite, or else they all terminate to the left, since there are 7 of these orbits and 7 projectives that might terminate them.

To show that the orbits are finite we consider the function $f_\omega$ defined before as $f_\omega(M):=\dim_k\Hom(P_\omega,M)$, which is additive on every mesh. The values of $f_\omega$ on the part of the quiver already computed are sufficient to compute $f_\omega$ everywhere, by additivity. These values are shown on the quiver diagram. If $\tau$-orbits were infinite the values shown would continue with negative values, which is not possible. We deduce that the quiver component is finite, and hence by a theorem of Auslander that $k\calP$ has finite representation type and this is the entire quiver. Furthermore $f_\omega$ takes the value 1 on all indecomposable projectives, so the quiver must terminate to the left where there is a pattern of 1s the same as the pattern of the projectives. The only time this happens before $f_\omega$ becomes negative is at the left of what is shown of the quiver, and hence the quiver must terminate at the left as shown. Because $\calP$ is isomorphic to its opposite poset, the quiver must terminate to the right with injectives in the same fashion.
\end{example}

\begin{example} \label{equivalent_to_different_rep_type}
We may compute the Auslander-Reiten quivers of both $D^b(k\calP)$ and of $k\calP$-mod by these means for many other posets $\calP$. For instance, the poset
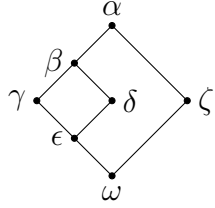
\begin{center}
\begin{tikzpicture}[xscale=.5,yscale=.5]
\draw (2,0)--(1,1)--(0,2)--(1,3)--(2,4)--(4,2)--(2,0);
\draw (1,1)--(2,2)--(1,3);
\draw[fill] (2,0) circle [radius=0.08];
\draw[fill] (1,1) circle [radius=0.08];
\draw[fill] (0,2) circle [radius=0.08];
\draw[fill] (1,3) circle [radius=0.08];
\draw[fill] (2,4) circle [radius=0.08];
\draw[fill] (4,2) circle [radius=0.08];
\draw[fill] (2,2) circle [radius=0.08];

\node [ below] at  (2,0)  {$\omega$};
\node [ left] at  (1,1)   {$\epsilon$};
\node [ right] at  (2,2)   {$\delta$};
\node [ left] at  (0,2)   {$\gamma$};
\node [ left] at  (1,3)   {$\beta$};
\node [ above] at  (2,4)   {$\alpha$};
\node [ right] at  (4,2)   {$\zeta$};
\end{tikzpicture}
\end{center}
has finite representation type but is derived equivalent to the path algebra of a tree with underlying diagram
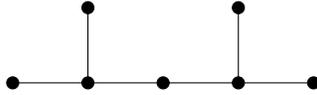
\begin{center}
\begin{tikzpicture}[scale=1]
\draw (0,0)--(1,0)--(2,0)--(3,0)--(4,0);
\draw (1,0)--(1,1); \draw (3,0)--(3,1);
\foreach \y in {0,1,2,3,4}
\draw[fill] (\y,0) circle [radius=0.08];
\draw[fill] (1,1) circle [radius=0.08];
\draw[fill] (3,1) circle [radius=0.08];
\end{tikzpicture}
\end{center}
which is Euclidean and of tame representation type

The poset
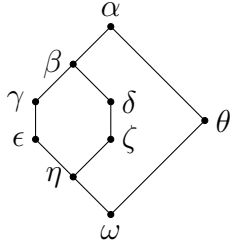
\begin{center}
\begin{tikzpicture}[xscale=.5,yscale=.5]
\draw (2,0)--(1,1)--(0,2)--(0,3)--(1,4)--(2,5)--(4.5,2.5)--(2,0);
\draw (1,1)--(2,2)--(2,3)--(1,4);
\draw[fill] (2,0) circle [radius=0.08];
\draw[fill] (1,1) circle [radius=0.08];
\draw[fill] (0,2) circle [radius=0.08];
\draw[fill] (0,3) circle [radius=0.08];
\draw[fill] (1,4) circle [radius=0.08];
\draw[fill] (2,5) circle [radius=0.08];
\draw[fill] (4.5,2.5) circle [radius=0.08];
\draw[fill] (2,2) circle [radius=0.08];
\draw[fill] (2,3) circle [radius=0.08];
\node [ below] at  (2,0)  {$\omega$};
\node [ left] at  (1,1)   {$\eta$};
\node [ right] at  (2,2)   {$\zeta$};
\node [ left] at  (0,2)   {$\epsilon$};
\node [ right] at  (2,3)   {$\delta$};
\node [ left] at  (0,3)   {$\gamma$};
\node [ left] at  (1,4)   {$\beta$};
\node [ above] at  (2,5)   {$\alpha$};
\node [ right] at (4.5,2.5) {$\theta$};
\end{tikzpicture}
\end{center}
is of infinite representation type and has a stable component of the Auslander-Reiten quiver of $k\calP$-mod that has Euclidean tree class of type $\tilde E_6$.  We leave the details to the reader.
\end{example}

We can extract from the above methods a criterion for a poset to be of finite representation type.

\begin{corollary}
\label{finite-rep-type-criterion}
Let $\calP$ be a poset satisfying the conditions of Proposition~\ref{inductive-clamping-proposition}, so that the component of the Auslander-Reiten quiver of $D^b(k\calP)$ containing $P_\alpha$ has tree class $T$. If, after removing the vertex of $T$ corresponding to $P_\alpha$, the resulting tree is a Dynkin diagram then $k\calP$ has finite representation type.
\end{corollary}

\begin{proof}
We know from Proposition~\ref{AR-quiver-subset-of-ZT} that the component of the Auslander-Reiten quiver of $k\calP$-mod that contains $P_\alpha$ is a subset of $\ZZ T$, and the infinite $\tau$-orbits to the left of the slice $S$ correspond to a subdiagram of $T$ that does not include the vertex for $P_\alpha$. This subdiagram must be a union of Dynkin trees. The function $f_\omega$ must be non-zero on these infinite orbits, since any adjacent finite $\tau$-orbit must connect to an infinite orbit via an irreducible morphism of the form $\Rad P_\sigma \to P_\sigma$ and $f_\omega$ is 1 on both these modules. We now conclude that $f_\omega$ must take negative values on some infinite $\tau$-orbit, if it exists (because there is no positive additive function on $\ZZ\Delta$ when $\Delta$ is Dynkin), which is not possible. From this we conclude that the Auslander-Reiten quiver is finite to the left, and by a dual argument it is finite to the right. Hence the Auslander-Reiten quiver of $k\calP$ has a finite component, and by a theorem of Auslander $k\calP$ has finite representation type.
\end{proof}

We note that the converse of Corollary~\ref{finite-rep-type-criterion} is false, as is shown by Example~\ref{AR-quiver-example}.

\section{$\mathcal{IC}$ posets of finite representation type}

The posets of finite representation type were determined by Loupias \cite{Lou}, and enumerated by Drozdowski and Simson \cite{DS}.  Here we illustrate how the techniques from the previous section can be used to determine the Auslander-Reiten quivers of $\mathcal{IC}$ posets of finite representation type.  In \cite{DS} the authors organized the posets of finite representation type into families where the $\mathcal{IC}$ posets of finite representation type are named Rys 28, 29, and 30.  We remark that while the families Rys 1, 2, and 12 are not necessarily $\mathcal{IC}$ posets, clamping theory can also be used to examine their Auslander-Reiten quivers.

In what follows we divide the $\mathcal{IC}$ posets of finite representation type into eight classes. For each class we give a slice of the Auslander-Reiten quiver of $k \mathcal P$-mod containing the projective module $P_{\alpha}$ that we denote by an open circle in our diagram.  While our focus is on the module category, we remark that, by Corollary~\ref{wing-copy-corollary}, this diagram is the tree class of an Auslander-Reiten component in the derived category $D^b(k\calP)$. 

For all of these algebras, the subquiver of their Auslander-Reiten quiver consisting of only the projective modules is a disjoint union of Dynkin diagrams.  For the classes (1)--(4), after removing the vertex corresponding to $P_{\alpha}$, the slice of the Auslander-Reiten quivers is a Dynkin diagram. We remark that these algebras can be seen to have finite representation type by a quick application of Corollary~\ref{finite-rep-type-criterion}.   The remaining classes (5)--(8) behave similarly to Examples~\ref{AR-quiver-example} and \ref{equivalent_to_different_rep_type}.  That is, the Auslander-Reiten quiver has some $\tau$-orbits that terminate quickly.  These orbits correspond to the boxed vertices in our diagrams.  After removing these vertices and the vertex corresponding to $P_{\alpha}$ the remaining $\tau$-orbits, given as filled circles, form a Dynkin diagram.  

We illustrate our notation on Example~\ref{AR-quiver-example}. There the slice for the Auslander-Reiten quiver of $k \mathcal P$-mod would be the tree  

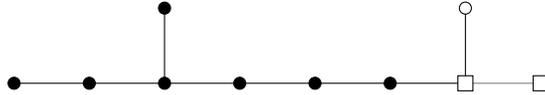
\begin{center}
 \begin{tikzpicture}[scale=1]
\phantom{\draw (8,0);}
\draw[black] (1,0)-- (2,0)--(3,0)--(4,0)-- (5,0)--(6,0)--(6.9,0);
\draw[gray] (7.1,0)--(7.9,0);
\draw(3,0)--(3,0.925);
\draw(7,.1)--(7,0.925);
\foreach \y in {1,2,3,4,5,6}
\draw[fill, black] (\y,0) circle [radius=0.08];
\draw (6.9,-.1) rectangle (7.1,.1);
\draw (7.9,-.1) rectangle (8.1,.1);
\draw[fill, black] (3,1) circle [radius=0.08];
\draw[black] (7,1) circle [radius=0.08];
\end{tikzpicture}
\end{center}



We now give the promised list of  $\mathcal{IC}$ posets of finite representation type.  

\begin{center}
\begin{tabular}{ |c|c| c| }
\hline
    Poset & Slice of Auslander-Reiten Quiver & Rys Label   \\ \hline
%
%
\begin{tikzpicture}[xscale=.5,yscale=.5]
\phantom{\draw (1,6);}
\draw (1,0)--(0,2.5)--(1,5)--(1,0)--(2,1)--(2,2) (2,3)--(2,4)--(1,5);
\draw[dotted] (2,2)--(2,3);
\foreach \x in {0,1}
\draw[fill] (\x,2.5) circle [radius=0.08];
\foreach \y in {1,2,3,4}
\draw[fill] (2,\y) circle [radius=0.08];
\draw[fill] (1,0) circle [radius=0.08];
\draw[fill] (1,5) circle [radius=0.08];
\draw[decorate, decoration = {brace, amplitude = 5pt, mirror}, xshift = 5pt, yshift = 0pt] (2,1)--(2,4) node[black, midway, xshift = 10pt]{\footnotesize $p$};
\end{tikzpicture}
& 
\begin{tikzpicture}[scale=.75]
\draw (0,0)--(1,0)--(2,0)--(3,0);
\draw[gray] (3,0) (4,0)--(5,0);
\draw[thick, dotted] (3.2,0)--(3.8,0);
\draw (1,0)--(1,0.925); \draw (1,0)--(1,-1);
\foreach \y in {0,1}
\draw[fill] (\y,0) circle [radius=0.08];
\foreach \x in {2,3,4,5}
\draw[gray, fill] (\x,0) circle [radius=0.08];
\draw[black] (1,1) circle [radius=0.08];
\draw[fill] (1,-1) circle [radius=0.08];
\draw[decorate, decoration = {brace, amplitude = 5pt, mirror}, xshift = 0pt, yshift = -4pt] (2,0)--(5,0) node[black, midway, yshift = -10pt, rotate = 0]{\footnotesize $p$};
\end{tikzpicture}
& 
\begin{tabular}{c} $28$ \\  \small{$p\geq 1 $} \end{tabular}\\ \hline
%
%
%
\begin{tikzpicture}[xscale=.5,yscale=.5]
\phantom{\draw (1,5);}
\draw (1,0)--(0,1.3)--(0,2.7)--(1,4)--(1,0)--(2,1) (2,2)--(2,3)--(1,4);
\draw[dotted] (2,1)--(2,2);
\draw[fill] (1,2) circle [radius=0.08];
\foreach \y in {1.3,2.7}
\draw[fill] (0,\y) circle [radius=0.08];
\foreach \y in {1,2,3}
\draw[fill] (2,\y) circle [radius=0.08];
\draw[fill] (1,0) circle [radius=0.08];
\draw[fill] (1,4) circle [radius=0.08];
\draw[decorate, decoration = {brace, amplitude = 5pt, mirror}, xshift = 5pt, yshift = 0pt] (2,1)--(2,3) node[black, midway, xshift = 10pt]{\footnotesize $p$};
\end{tikzpicture} & 
%
\begin{tikzpicture}[scale=.75]
\draw (-1,0)--(0,0)--(1,0)--(2,0)--(3,0);
\draw[gray] (2,0)--(3,0) (4,0);
\draw[thick, dotted] (3.2,0)--(3.8,0);
\draw (1,0)--(1,0.925); \draw (1,0)--(1,-1);
\foreach \y in {-1,0,1}
\draw[fill] (\y,0) circle [radius=0.08];
\foreach \x in {2,3,4}
\draw[gray, fill] (\x,0) circle [radius=0.08];
\draw[black] (1,1) circle [radius=0.08];
\draw[fill] (1,-1) circle [radius=0.08];
\draw[decorate, decoration = {brace, amplitude = 5pt, mirror}, xshift = 0pt, yshift = -4pt] (2,0)--(4,0) node[black, midway, yshift = -10pt, rotate = 0]{\footnotesize $p$};
\end{tikzpicture} & 
\begin{tabular}{c} 28 \\ \small{$2\leq p \leq 4$ } \end{tabular} \\ \hline
%
%
%
\begin{tikzpicture}[xscale=.5,yscale=.5]
\phantom{\draw (1,7);}
\draw (0,5)--(1,6)--(2,4)--(2,3) (2,2)--(2,1)--(1,0);
\draw (1,0)--(0,1)--(0,2) (0,3)--(0,5);
\draw[dotted] (0,2)--(0,3) (2,2)--(2,3);
\foreach \y in {1,2,3,4,5}
\draw[fill] (0,\y) circle [radius=0.08];
\foreach \y in {1,2,3,4}
\draw[fill] (2,\y) circle [radius=0.08];
\draw[fill] (1,0) circle [radius=0.08];
\draw[fill] (1,6) circle [radius=0.08];
\draw[decorate, decoration = {brace, amplitude = 5pt}, xshift=-5pt, yshift=-0pt] (0,1)--(0,5) node[black, midway, xshift = -10pt]{\footnotesize $n$};
\draw[decorate, decoration = {brace, amplitude = 5pt, mirror}, xshift = 5pt, yshift = 0pt] (2,1)--(2,4) node[black, midway, xshift = 10pt]{\footnotesize $p$};
\end{tikzpicture} & 
\begin{tikzpicture}[scale=.75]
\draw[gray] (0,0)--(1,0) (2,0)--(3,0);
\draw[thick, dotted] (1.2,0)--(1.8,0);
\draw[gray] (3,0)--(4,0) (5,0)--(6,0);
\draw[thick, dotted] (4.2,0)--(4.8,0);
\draw(3,0)--(3,0.925);
\draw[black, fill] (3,0) circle [radius=0.08];
\foreach \y in {0,1,2}
\draw[fill, gray] (\y,0) circle [radius=0.08];
\foreach \x in {4,5,6}
\draw[fill,gray] (\x,0) circle [radius=0.08];
\draw[black] (3,1) circle [radius=0.08];
\draw[decorate, decoration = {brace, amplitude = 5pt, mirror}, xshift=0pt, yshift=-4pt] (0,0)--(2,0) node[black, midway, yshift = -10pt, rotate =0]{\footnotesize $n$};
\draw[decorate, decoration = {brace, amplitude = 5pt, mirror}, xshift = 0pt, yshift = -4pt] (4,0)--(6,0) node[black, midway, yshift = -10pt, rotate = 0]{\footnotesize $p$};
\end{tikzpicture} 
& 
\begin{tabular}{c} $29$  \\ \small{$n \geq 0$}  \\ \small{$p\geq 0$} \end{tabular}
\\ \hline
%
%

%
\begin{tikzpicture}[xscale=.5,yscale=.5]
\phantom{\draw (1,8);}
\draw (2,0)--(1,1)--(1,1.75) (1,2.5)--(0,3.5)--(1,4.5) (1,5.25)--(1,6)--(2,7)--(4,5)--(4,4) (4,3)--(4,2)--(2,0);
\draw[dotted] (1,1.75)--(1,2.5) (1,4.5)--(1,5.25) (4,4)--(4,3);
\draw (1,2.5)--(2,3.5)--(1,4.5);
\draw[fill] (2,0) circle [radius=0.08];
\draw[fill] (1,1) circle [radius=0.08];
\draw[fill] (1,1.75) circle [radius=0.08];
\draw[fill] (1,2.5) circle [radius=0.08];
\draw[fill] (0,3.5) circle [radius=0.08];
\draw[fill] (1,4.5) circle [radius=0.08];
\draw[fill] (1,5.25) circle [radius=0.08];
\draw[fill] (1,6) circle [radius=0.08];
\draw[fill] (2,7) circle [radius=0.08];
\foreach \y in {2,3,4,5}
\draw[fill] (4,\y) circle [radius=0.08];
\draw[fill] (2,3.5) circle [radius=0.08];
\draw[decorate, decoration = {brace, amplitude = 5pt, mirror}, xshift = 5pt, yshift = 0pt] (4,2)--(4,5) node[black, midway, xshift = 10pt]{\footnotesize $p$};
\draw[decorate, decoration = {brace, amplitude = 5pt}, xshift = -5pt, yshift = 0pt] (1,1)--(1,2.5) node[black, midway, xshift = -10pt]{\footnotesize $r$};
\draw[decorate, decoration = {brace, amplitude = 5pt}, xshift = -5pt, yshift = 0pt] (1,4.5)--(1,6) node[black, midway, xshift = -10pt]{\footnotesize $n$};
\end{tikzpicture} & 
%
\begin{tikzpicture}[scale=.75]
\draw (0,0)--(1,0)--(2,0)--(3,0);
\draw (6,0)--(7,0);
\draw[gray] (3,0)--(4,0) (5,0)--(6,0);
\draw[thick, dotted] (4.2,0)--(4.8,0);
\draw[gray] (7,0)--(8,0) (9,0)--(10,0);
\draw[thick, dotted] (8.2,0)--(8.8,0);
\draw (1,0)--(1,1); \draw (7,0)--(7,0.925);
\foreach \y in {0,1,2,7}
\draw[fill] (\y,0) circle [radius=0.08];
\foreach \x in {3,4,5,6}
\draw[gray, fill] (\x,0) circle [radius=0.08];
\foreach \z in {8,9,10}
\draw[gray, fill] (\z,0) circle [radius=0.08];
\draw[fill] (1,1) circle [radius=0.08];
\draw (7,1) circle [radius=0.08];
\draw[decorate, decoration = {brace, amplitude = 5pt, mirror}, xshift = 0pt, yshift = -4pt] (8,0)--(10,0) node[black, midway, yshift = -10pt, rotate = 0]{\footnotesize $p$};
\draw[decorate, decoration = {brace, amplitude = 5pt, mirror}, xshift = 0pt, yshift = -4pt] (3,0)--(6,0) node[black, midway, yshift = -10pt, rotate = 0]{\footnotesize $n+r$};
\end{tikzpicture} & 
\begin{tabular}{c} $30$ \\ \small{$ n\geq 1$} \\ \small{$r \geq 1$}  \\ \small{$p \geq 1$}  \end{tabular}\\ \hline
%
%
\begin{tikzpicture}[scale=.5]
\phantom{\draw (1,6);}
\draw(1,0)--(0,1)--(1,2.5)--(0,4)  (0,1)--(-1,2) (-1,3)--(0,4)--(1,5)--(3,2.5)--(1,0);
\draw[dotted] (-1,2)--(-1,3);
\foreach \y in {0,2.5,5}
\draw[fill] (1,\y) circle [radius=0.08];
\foreach \y in {1,4}
\draw[fill] (0,\y) circle [radius=0.08];
\foreach \y in {2,3}
\draw[fill] (-1,\y) circle [radius=0.08];
\draw[fill] (3,2.5) circle [radius=0.08];
\draw[decorate, decoration = {brace, amplitude = 5pt}, xshift = -5pt, yshift = 0pt] (-1,2)--(-1,3) node[black, midway, xshift = -10pt]{\footnotesize $q$};
\end{tikzpicture} & 
\begin{tikzpicture}[scale=1]
\draw (0.1,0)--(1,0)--(2,0)--(3,0)--(4,0);
\draw[gray] (4,0)--(5,0) (6,0)--(7,0);
\draw[thick, dotted] (5.2,0) -- (5.8,0);
\draw (1,0)--(1,0.925); \draw (3,0)--(3,1);
\foreach \y in {1,2,3}
\draw[fill] (\y,0) circle [radius=0.08];
\foreach \x in {4,5,6,7}
\draw[gray, fill] (\x,0) circle [radius=0.08];
\draw (-.1,-.1) rectangle (.1,.1);
\draw (1,1) circle [radius=0.08];
\draw[fill] (3,1) circle [radius=0.08];
\draw[decorate, decoration = {brace, amplitude = 5pt, mirror}, xshift = 0pt, yshift = -4pt] (4,0)--(7,0) node[black, midway, yshift = -10pt, rotate = 0]{\footnotesize $q$};
\end{tikzpicture} & 
\begin{tabular}{c} 30 \\  \small{$1\leq q \leq 4 $} \\ \end{tabular} 
\\ \hline
\end{tabular}
\end{center} 
%
%
%
\begin{center}
\begin{tabular}{ |c|c| c| }
\hline
Poset & Slice of Auslander-Reiten Quiver & Rys Label   \\ \hline
%
\begin{tikzpicture}[scale=.5]
\phantom{\draw (1,6);}
\draw(1,0)--(0,1)--(1,2.5)--(0,4)  (0,1)--(-1,2)--(-1,3)--(0,4)--(1,5)--(2.5,3.5)  (2.5,1.5)--(1,0);
\draw[dotted] (2.5,3.5)--(2.5,1.5);
\foreach \y in {0,2.5,5}
\draw[fill] (1,\y) circle [radius=0.08];
\foreach \y in {1,4}
\draw[fill] (0,\y) circle [radius=0.08];
\foreach \y in {2,3}
\draw[fill] (-1,\y) circle [radius=0.08];
\foreach \y in {1.5,3.5}
\draw[fill] (2.5,\y) circle [radius=0.08];
\draw[decorate, decoration = {brace, amplitude = 5pt, mirror}, xshift = 5pt, yshift = 0pt] (2.5,1.5)--(2.5,3.5) node[black, midway, xshift = 10pt]{\footnotesize $p$};
\end{tikzpicture} & 
\begin{tikzpicture}[scale=1]
\draw[gray] (0.1,0)--(1,0) (2,0)--(3,0);
\draw[thick, dotted] (1.2,0)--(1.8,0);
\draw (3,0)--(4,0)--(5,0)--(6,0)--(7,0);
\draw (3,0)--(3,.925); \draw (5,0)--(5,1);
\foreach \y in {1,2}
\draw[fill, gray] (\y,0) circle [radius=0.08];
\foreach \x in {3,4,5,6,7}
\draw[fill] (\x,0) circle [radius=0.08];
\draw[gray] (-.1,-.1) rectangle (.1,.1);
\draw (3,1) circle [radius=0.08];
\draw[fill] (5,1) circle [radius=0.08];
\draw[decorate, decoration = {brace, amplitude = 5pt, mirror}, xshift = 0pt, yshift = -4pt] (0,0)--(2,0) node[black, midway, yshift = -10pt, rotate = 0]{\footnotesize $p$};
\end{tikzpicture} & 
\begin{tabular}{c} 30\\  \small{$1 \leq p\leq 3$} \end{tabular}
\\ \hline
%
%
%
%
\begin{tikzpicture}[xscale=.5,yscale=.5]
\phantom{\draw (1,8);}
\draw (2,0)--(1,1)--(1,2)--(0,3)--(0,4)--(1,5)--(1,6)--(2,7)--(4,3.5)--(2,0);
\draw (1,2)--(2,3.5)--(1,5);
\draw[fill] (2,0) circle [radius=0.08];
\draw[fill] (1,1) circle [radius=0.08];
\draw[fill] (1,2) circle [radius=0.08];
\draw[fill] (0,3) circle [radius=0.08];
\draw[fill] (0,4) circle [radius=0.08];
\draw[fill] (1,5) circle [radius=0.08];
\draw[fill] (1,6) circle [radius=0.08];
\draw[fill] (2,7) circle [radius=0.08];
\draw[fill] (4,3.5) circle [radius=0.08];
\draw[fill] (2,3.5) circle [radius=0.08];
\end{tikzpicture}  &
\begin{tikzpicture}[scale=1]
\draw (0.1,0)--(.9,0);
\draw (1.1,0)--(2,0)--(3,0)--(4,0)--(5,0)--(6,0)--(7,0);
\draw (1,0.1)--(1,.925); \draw (5,0)--(5,1);
\foreach \y in {2,3,4,5,6,7}
\draw[fill] (\y,0) circle [radius=0.08];
\draw (-.1,-.1) rectangle (.1,.1);
\draw (.9,-.1) rectangle (1.1,.1);
\draw[fill] (5,1) circle [radius=0.08];
\draw (1,1) circle [radius=0.08];
\end{tikzpicture} &
\begin{tabular}{c} 30\\  \end{tabular}
\\  \hline
%
%
%
\begin{tikzpicture}[xscale=.5,yscale=.5]
\phantom{\draw (1,9);}
\draw (2,-1)--(1,0)--(1,2)--(0,3)--(0,4)--(1,5)--(1,7)--(2,8)--(4,3.5)--(2,-1);
\draw (1,2)--(2,3.5)--(1,5);
\draw[fill] (2,-1) circle [radius=0.08];
\draw[fill] (1,0) circle [radius=0.08];
\draw[fill] (1,1) circle [radius=0.08];
\draw[fill] (1,2) circle [radius=0.08];
\draw[fill] (0,3) circle [radius=0.08];
\draw[fill] (0,4) circle [radius=0.08];
\draw[fill] (1,5) circle [radius=0.08];
\draw[fill] (1,6) circle [radius=0.08];
\draw[fill] (1,7) circle [radius=0.08];
\draw[fill] (2,8) circle [radius=0.08];
\draw[fill] (4,3.5) circle [radius=0.08];
\draw[fill] (2,3.5) circle [radius=0.08];
\end{tikzpicture}&
\begin{tikzpicture}[scale=1]
\draw (0.1,0)--(.9,0);
\draw (1.1,0)--(1.9,0);
\draw (2.1,0)--(3,0)--(4,0)--(4,0)--(5,0)--(6,0)--(7,0)--(8,0)--(9,0);
\draw (1,0.1)--(1,.925); \draw (7,0)--(7,1);
\foreach \y in {3,4,5,6,7,8,9}
\draw[fill] (\y,0) circle [radius=0.08];
\draw[fill] (7,1) circle [radius=0.08];
\draw (1,1) circle [radius=0.08];
\draw (-.1,-.1) rectangle (.1,.1);
\draw (.9,-.1) rectangle (1.1,.1);
\draw (1.9,-.1) rectangle (2.1,.1);
\end{tikzpicture}&
\begin{tabular}{c} 30\\   \end{tabular} \\
\hline
%
%
\begin{tikzpicture}[xscale=.5,yscale=.5]
\phantom{\draw (1,8);}
\draw (2,0)--(1,1)--(1,2)--(0,3)--(0,4)--(1,5)--(1,6)--(2,7)--(4,5)--(4,2)--(2,0);
\draw (1,2)--(2,3.5)--(1,5);
\draw[fill] (2,0) circle [radius=0.08];
\draw[fill] (1,1) circle [radius=0.08];
\draw[fill] (1,2) circle [radius=0.08];
\draw[fill] (0,3) circle [radius=0.08];
\draw[fill] (0,4) circle [radius=0.08];
\draw[fill] (1,5) circle [radius=0.08];
\draw[fill] (1,6) circle [radius=0.08];
\draw[fill] (2,7) circle [radius=0.08];
\draw[fill] (4,5) circle [radius=0.08];
\draw[fill] (4,2) circle [radius=0.08];
\draw[fill] (2,3.5) circle [radius=0.08];
\end{tikzpicture}&
\begin{tikzpicture}[scale=1]
\draw (-.9,0)--(-.1,0) (.1,0)--(2,0)--(3,0)--(4,0)--(5,0)--(6,0)--(7,0);
\draw (1,0)--(1,0.925); \draw (5,0)--(5,1);
\foreach \y in {1,2,3,4,5,6,7}
\draw[fill] (\y,0) circle [radius=0.08];
\foreach \x in {-1,0}
\draw[fill] (5,1) circle [radius=0.08];
\draw (1,1) circle [radius=0.08];
\draw (-1.1,-.1) rectangle (-.9,.1);
\draw (-.1,-.1) rectangle (.1,.1);
\end{tikzpicture}&
\begin{tabular}{c} 30\\  \end{tabular} \\
\hline
\end{tabular}
\end{center}



\begin{thebibliography}{99}

\bibitem{ASS} I. Assem, D. Simson and A. Skowronski, \textit{Elements of the representation theory of associative algebras Vol. 1: Techniques of representation theory,} London Mathematical Society Student Texts \textbf{65}, Cambridge University Press, Cambridge, 2006.

\bibitem{ARS} M. Auslander, I. Reiten and S.O. Smal\o, \textit{Representation theory of Artin algebras,} Cambridge studies in advanced mathematics \textbf{36}, Cambridge University Press, Cambridge, 1995.

\bibitem{Cha} F. Chapoton, \textit{On the Categories of Modules Over the Tamari Posets} pp. 269--280 in: F. M{\"u}ller-Hoissen,
J. M.  Pallo, and J. Stasheff, Associahedra, Tamari Lattices and Related Structures: Tamari Memorial Festschrift, Springer Basel, (2012). 


\bibitem{DS} G. Drozdowski and D. Simson,  \textit{Remarks on posets of finite representation type.} (unpublished?)  Accessed at: http://www-users.mat.umk.pl/~simson/DrozdowskiSimson1978.pdf

\bibitem{Hap} D. Happel, \textit{Triangulated Categories in the Representation Theory of Finite-Dimensional Algebras,} London
Math. Soc. Lecture Note Ser. \textbf{119}, Cambridge University Press, Cambridge, 1988.

\bibitem{HV} D. Happel, D. Vossieck, \textit{Minimal algebras of infinite representation type with preprojective component}, Manuscripta Math. \textbf{42}, (1983), 221--243.

\bibitem{HI} M. Herschend and O. Iyama, \textit{$n$-representation-finite algebras and twisted fractionally Calabi-Yau algebras,}
Bull. Lond. Math. Soc. \textbf{43} (2011), 449--466. 

\bibitem{KLM} D. Kussin, H. Lenzing, and H. Meltzer, \textit{Triangle singularities, ADE-chains, and weighted projective lines}, Advances in Mathematics, \textbf{237}, (2013), 194--251.

\bibitem{Lad} S. Ladkani, \textit{Which canonical algebras are derived equivalent to incidence algebras of posets} Comm. Algebra \textbf{36} (2008), 4599--4606.

\bibitem{Lou} M. Loupias, \textit{Indecomposable representations of finite ordered sets}. Representations of algebras (Proc. Internat. Conf., Carleton Univ., Ottawa, Ont., 1974), pp. 201-209. Lecture Notes in Math., Vol. 488, Springer, Berlin, (1975).

\bibitem{Mit1} B. Mitchell, \textit{Theory of categories}. Pure and Applied Mathematics, Vol. XVII Academic Press, New York-London, (1965).

\bibitem{Mit2} B. Mitchell, \textit{On the dimension of objects and categories. II. Finite ordered sets}, J. Algebra \textbf{9} (1968), 341--368.

\bibitem{Sch} S. Scherotzke, \textit{Finite and bounded Auslander-Reiten components in the derived category,} J. Pure Appl. Algebra \textbf{215} (2011), 232--241.

\bibitem{SS} D. Simson, A. Skowro\'nski, \textit{Elements of the Representation Theory of Associative Algebras 2: Tubes and Concealed Algebras of Euclidean Type}, Cambridge University Press, (2007).

\bibitem{Web1} P.J. Webb, \textit{An introduction to the representations and cohomology of categories,}
pp. 149-173 in: M. Geck, D. Testerman and J. Th\'evenaz (eds.), Group Representation Theory, EPFL Press (Lausanne) (2007). 

\bibitem{Web2} P.J. Webb, \textit{Bilinear forms on Grothendieck groups of triangulated categories,} preprint.

\end{thebibliography}
\end{document}